\documentclass[final]{siamltex}

\usepackage{amssymb}  
\usepackage{amsmath} 
\usepackage{tabularx}
\usepackage{thmtools}
\usepackage{url}

\usepackage{ifpdf}
\ifpdf
    \usepackage{graphicx}
    \usepackage{epstopdf}
\else
    \usepackage{graphicx}

\fi

\usepackage[dvipsnames,usenames]{color}
\usepackage[colorlinks,%
linkcolor=BrickRed,%
filecolor =BrickRed,%
citecolor=RoyalPurple,%
]{hyperref}

\def\notes#1{\typeout{Read the notes!}}

\def\barDelta{\bar \Delta}
\def\epsy{\varepsilon}
\def\pertu{\nu} 

\newcommand{\ud}{\,\mathrm{d}}

\def\varble{\, \cdot\,}
\def\FRAC#1#2#3{\genfrac{}{}{}{#1}{#2}{#3}}
\def\half{{\mathchoice{\FRAC{1}{1}{2}}%
{\FRAC{2}{1}{2}}%
{\FRAC{3}{1}{2}}%
{\FRAC{4}{1}{2}}}}

\def\bbbone{{\mathchoice {\rm 1\mskip-4mu l} {\rm 1\mskip-4mu l}
{\rm 1\mskip-4.5mu l} {\rm 1\mskip-5mu l}}}
\def\ind{\bbbone}

\def\Sec#1{Sec.~\ref{#1}}

\def\Prop#1{Prop.~\ref{#1}}

\def\gen{{\cal D}}

\def\qed{\hspace*{\fill}~\QED\par\endtrivlist\unskip}

\newcommand{\tr}{\mbox{\rm tr}}

\def\Inov{I} 

\def\v{{\sf K}}
\def\w{{\sf \Omega}}
\def\Expect{{\sf E}}
\def\E{{\sf E}}

\def\clZ{{\cal Z}}

\def\w{{\sf \Omega}}

\def\Sec#1{Sec~\ref{#1}}

\def\smopot{{\cal G}}

\def\QED{\mbox{\rule[0pt]{1.3ex}{1.3ex}}}

\newcounter{anum}

\def\R{\mathbb R}

\newtheorem{remark}{Remark}

\title{
Poisson's Equation in Nonlinear Filtering
}

\author{Richard S. Laugesen
\thanks{R.~S.~Laugesen is with the Department of Mathematics at University of
Illinois at Urbana-Champaign (UIUC)
({\tt\small laugesen@illinois.edu})}
\and Prashant G.~Mehta
\thanks{P.~G.~Mehta is with the Coordinated Science Laboratory and
the Department of Mechanical Science and Engineering at UIUC
({\tt\small mehtapg@illinois.edu})}
\and \linebreak Sean P.~Meyn%
\thanks{S.~P.~Meyn is with the Department of Electrical and Computer
Engineering at University of Florida at Gainesville
({\tt\small meyn@ufl.edu})}%
\and Maxim Raginsky%
\thanks{M.~Raginsky is with the Department of Electrical and Computer
Engineering and the Coordinated Science Laboratory at UIUC ({\tt\small maxim@illinois.edu})}
}

\begin{document}
\maketitle

\begin{abstract}
The aim of this paper is to provide a variational interpretation of the
nonlinear filter in continuous time.  A time-stepping procedure is introduced, 
consisting of successive minimization problems in the space of probability densities.  The weak
form of the nonlinear filter is derived via analysis of the first-order optimality conditions for these problems.  The derivation shows the nonlinear filter dynamics may be regarded as a gradient flow,
or a steepest descent, for a certain 
energy functional
with respect to the Kullback--Leibler divergence.

\smallskip

The second part of the paper is concerned with derivation of the
feedback particle filter algorithm, based again on the analysis of the
first variation.  The algorithm is shown to be exact.  That is, the
posterior distribution of the particle matches exactly the true
posterior, provided the filter is initialized with the true prior. 
\end{abstract}

\section{Introduction}

The goal of this paper is to gain insight into the equations arising in nonlinear filtering, as well as into the \textit{feedback particle filter} introduced in recent research.   To expose the main ideas, it is useful to restrict our attention to the following special case in which the state evolution is constant:  
\begin{subequations}
\begin{align}
\ud X_t &= 0,
\label{eqn:Signal_Process}
\\
\ud Z_t &= h(X_t)\ud t + \ud W_t,
\label{eqn:Obs_Process}
\end{align}
\end{subequations}
where $X_t\in\R^d$ is the state at time $t$, 
$Z_t \in\R^1$ is the
observation process, $h(\varble)$ is a $C^1$
function, and $\{W_t\}$ is a standard
Wiener process.  The state is constant, and has initial condition distributed as $X_0\sim p_0^*$.   Unless otherwise noted, the stochastic
differential equations (SDEs) are expressed in It\^{o} form. Also, unless noted otherwise, all probability distributions are assumed to be absolutely continuous with respect to the Lebesgue measure, and therefore will be identified with their densities.

\smallskip

The objective of the filtering problem is to estimate the
posterior distribution of $X_t$ given the history $\clZ_t :=
\sigma(Z_s:  s \le t)$. The posterior is denoted by $p^*$, so
that for any measurable set $A\subset \R^d$,
\begin{equation}
\int_{A} p^*(x,t)\, \ud x   = {\sf P} \{ X_t \in A\mid \clZ_t \}.
\nonumber
\end{equation}
The evolution of  $p^*(x,t)$ is described by the Kushner--Stratonovich
(K-S) partial differential equation
\begin{equation}
\ud p^\ast = ( h-\hat{h} )(\ud Z_t - \hat{h} \ud t)p^\ast,
\label{eqn:Kushner_eqn}
\end{equation}
with initial condition $p_0^*$, where $ \hat{h}_t = \int h(x) p^*(x,t) \ud
x$. The theory of nonlinear
filtering is described in the classic monograph~\cite{kal80}.

\smallskip

Although our analysis is restricted to a particular model with a static
state process, it can be extended to broader classes of filtering
problems, subject to technical conditions discussed in
Remark~\ref{r:spectralBdds}.  The main technical condition concerns
the existence of a solution and certain a priori bounds for Poisson's
equation that also arises in simulation and optimization theory for
Markov models~\cite{glymey96a,MT}.  
For the model considered in this paper, bounds
are obtained based on a \textit{Poincar\'{e}}, or \textit{spectral gap}, \textit{inequality} (see
the bound  {\bf  PI}($\lambda_0$) in Assumption~A2).

\medskip

The contributions of this paper are two-fold: One, to show that the
dynamics of the K-S equation are a gradient flow for a certain
variational problem, with respect to the Kullback--Leibler divergence. Two,
  the variational problem is used  to derive the feedback particle filter,
first introduced in~\cite{taoyang_TAC12} (see
also~\cite{yanmehmey11,yanmehmey11b,yanlaumehmey12}).

\smallskip

\sloppypar The first part of the paper concerns the construction of the gradient flow.
The analysis is
inspired by the optimal transportation literature -- in particular,
the work of Otto and co-workers on the variational interpretation of
the Fokker--Planck--Kolmogorov equation~\cite{Jordan99thevariational}.  
The construction described in \Sec{sec:TS_prcdr} begins with a discrete-time recursion based on the successive solution of 
minimization problems involving the 
so-called forward variational representation of the elementary Bayes'
formula (see Mitter and Newton~\cite{MitterNewton04}).  Lemma~\ref{lem:EL} describes the first order optimality
condition for the variational problem at each time-step.

\smallskip

In the continuous-time limit, these first-order conditions yield the nonlinear
filter~\eqref{eqn:Kushner_eqn}, as described in the proof of
Theorem~\ref{thm:KS}.  The construction shows that the
dynamics of the nonlinear filter may be
regarded as a gradient flow, or a steepest descent, for a certain
energy functional
(``information value of the observation'' according to~\cite{MitterNewton04}) 
with respect to the Kullback-Leibler divergence pseudo-metric.

\smallskip

The feedback particle filter algorithm is obtained using similar analysis in
 \Sec{sec:FPF}.  This filter
is a controlled system, where the control  is obtained via 
consideration of the first order optimality conditions for the
variational problem.  Theorem~\ref{thm:exact} shows that the filter is exact,
i.e., the posterior distribution of the particle matches exactly the
true posterior $p^*$, provided the filter is initialized with the true
prior. 

\smallskip

The remainder of this paper is organized as follows. The time-stepping
procedure is introduced in \Sec{sec:TS_prcdr}, and properties of its
solution established.  The gradient flow
result -- convergence is the solution of the time-stepping procedure
to weak solution of the K-S equation~\eqref{eqn:Kushner_eqn} --
appears in \Sec{sec:gradient_flow}.  The feedback particle filter
algorithm appears in \Sec{sec:FPF}.  

\medskip

\paragraph{Notation:} $C^k$ is used to denote the space of $k$-times
continuously differentiable functions; $C^k_c$ denotes the subspace of
functions with compact support.  
$L^\infty$ is used to denote the
space of functions that are bounded a.e.\ (Lebesgue).

The space of probability densities with finite second moment is denoted
\begin{equation}
{\cal P} \doteq \left\{ \rho:\R^d\rightarrow [0,\infty) \,
    \text{measurable} \,\Big{\vert} \, \int_{\R^d} \rho(x) \ud x =
    1,\;\int x^2 \rho(x) \ud x < \infty 
\right\} .
\label{e:clP}
\end{equation}
$L^2(\R^d;\rho)$ denotes the Hilbert space of functions
on $\R^d$ that are square-integrable with respect to density
$\rho$;
$H^k(\R^d;\rho)$   denotes the Hilbert space of
functions whose first $k$ derivatives (defined in the weak or distributional sense)
are in $L^2(\R^d;\rho)$, and $H_0^1(\R^d;\rho) \doteq \{ \phi \in H^1(\R^d;\rho) \,\Big{\vert} \, \int \phi(x) \rho(x) \ud x = 0 \}$. 

For a function $f$, $\nabla f = \frac{\partial }{\partial x_i} f$ is
used to denote the gradient 
and $D^2 f = \frac{\partial^2 }{\partial x_i x_j} f$ is used to denote
the Hessian.  The derivatives are interpreted in the weak sense.  \qed

\newpage

\section{Time-Stepping Procedure}
\label{sec:TS_prcdr}

The time-stepping procedure involves a sequence of minimization
problems in the space of probability densities $\cal P$. 
We consider a finite time interval $[0,T]$ with an associated
discrete-time sequence $\{t_0, t_1, t_2,\hdots,t_N\}$ of sampling instants, with $t_0=0 < t_1 < \ldots < t_N = T$.
The corresponding increments are given by $\Delta t_n \doteq t_n - t_{n-1}, n = 1,\ldots,N$. 

A realization of the stochastic
process $Z_t$, the solution of SDE~\eqref{eqn:Obs_Process}, sampled at
discrete times is written as $\{Z_0, Z_1, Z_2,\hdots,Z_N\}$.  We use $\Delta
Z_n \doteq Z_n-Z_{n-1}$ to define the discrete-time observation process, and let
\[
Y_n \doteq \frac{\Delta Z_n}{\Delta t_n}.
\]
In discrete time, $Y_n$ 
is
 viewed as the observation made at time
$t_n$.  We eventually let $N\to\infty$ and simultaneously let $\barDelta_N\to 0$, where
\begin{equation}
\barDelta_N =\max\{ \Delta t_n :  n\le N \}\, .
\label{e:DeltaN}
\end{equation}

The elementary Bayes theorem is used to obtain the posterior
distribution, 
expressed recursively as
\begin{align}
\rho_0(x) & = p_0^*(x),\\
\rho_n(x) & = \frac{\rho_{n-1} (x) \exp(-\phi_n(x))}{\int \rho_{n-1}
  (y) \exp(-\phi_n(y)) \ud y}, 
\label{eq:Bayes_with_Yn}
\end{align}
where $\phi_n(x)\doteq\frac{\Delta t_n}{2} (Y_n - h(x))^2$. Note that the $\{\rho_n\}$ are \textit{random} probability measures since they depend on the discrete-time process $\{Z_n\}$. In particular, $\rho_n$ is measurable w.r.t.\ $\sigma(Z_i : i = 0,\ldots,n)$. This observation should be kept in mind when dealing with various parameters associated with the $\rho_n$, e.g., norm bounds for functions in $L^p(\R^d; \rho_n)$.

\medskip

The variational formulation of the Bayes recursion is the following
{\em time-stepping procedure}: Set $\rho_0 = p_0^* \in {\cal P}$  and
inductively define $\{\rho_n\}_{n=1}^{N} \subset {\cal P}$ by taking
$\rho_n \in {\cal P}$ to minimize the functional 
\begin{equation}
I_n(\rho)\doteq D(\rho \mid \rho_{n-1}) + \frac{\Delta t_n}{2} \int
\rho(x) (Y_n - h(x))^2 \ud x,
\label{eqn:obj_fn}
\end{equation}
where $D$ denotes the relative entropy or Kullback--Leibler divergence,
\[
D(\rho \mid \rho_{n-1}) =
\int \rho(x)
\ln \Bigl(\frac{\rho(x)}{\rho_{n-1} (x)} \Bigr) \ud x.
\]

The proof that $\rho_n$, as defined in~\eqref{eq:Bayes_with_Yn}, is in
fact the minimizer is straightforward: By Jensen's formula, $I_n(\rho)
\ge - \ln(\int\rho_{n-1}(y)\exp(-\phi_n(y))\ud y)$ with equality if
and only if $\rho=\rho_n$.   The optimizer $\rho_n$ is in fact the
``twisted distribution'' that arises in the theory of large deviations for empirical
means \cite{demzei98a}.   Although the optimizer is known,  a careful
look at the first order optimality equations associated with $\rho_n$
leads to i) the nonlinear filter~\eqref{eqn:Kushner_eqn} for evolution
of the posterior (in~\Sec{sec:gradient_flow}), and ii) a particle
filter algorithm for
approximation of the posterior (in~\Sec{sec:FPF}).

\medskip

Throughout the paper, the following assumptions are made for the prior
distribution $p_0^*$ and for function $h$:
\begin{romannum}
\item[{\textbf{\textit{Assumption~A1}}}]
The probability density $p_0^*\in {\cal P}$ is of the form $p_0^*(x) = e^{-\smopot_0(x)}$,
where $\smopot_0\in C^2$, 
$D^2 \smopot_0 \in L^\infty$,  and
$|\nabla \smopot_0|(x) \rightarrow \infty $ as
$|x|\rightarrow \infty$.
\notes{spm: reviewer was annoyed at redundancy: $\nabla \smopot_0(x) = O(|x|)$}

\item[{\textbf{\textit{Assumption~A2}}}]
The function $h \in C^2$ with $h,\nabla h, D^2h \in L^{\infty}$. 
\end{romannum}

Under assumption~A1, the density $\rho_0=p_0^*$
is known to admit a spectral gap (or Poincar\'{e} inequality) 
\cite{BakryPoincare}: That
is, 
for some $\lambda_0>0$, and
for all functions $f \in H^1(\R^d;\rho_0)$ with $\int \rho_0 f \ud x = 0$,
$$
\int |f(x)|^2 \rho_0(x) \ud x \le \frac{1}{\lambda_0}
	 \int		|\nabla f(x)|^2 \rho_0(x) \ud x. \eqno{[\text{{\bf  PI}($\lambda_0$)}]}
$$

The following proposition shows that the minimizers all admit a
uniform spectral gap. 
 The proof appears in the Appendix~\ref{proof:prop:exist_uniq}. 

\begin{proposition}
\label{prop:exist_uniq}
Under Assumption~(A1)-(A2),

\noindent (i) The minimizer $\rho_n$ is of the form $\rho_n = e^{-\smopot_n(x)}$,
where $\smopot_n\in C^2$.  These functions admit the following bounds, uniformly in $n$: 
$\nabla \smopot_n (x) = O(|x|)$, 
$|\nabla \smopot_n| (x) \rightarrow \infty$ as
$|x|\rightarrow \infty$, and $D^2 \smopot_n \in L^\infty$.  
\notes{spm: uniformity ok?  Did I miss something?}

\noindent (ii)  Suppose $f\in L^2(\R^d;\rho_{n-1})$.
Then $f\in
L^2(\R^d;\rho_{n})$ with
\begin{equation}
\int \rho_n(x) |f(x)|^2 \ud x \le C \exp(\alpha |\Delta Z_n|) \int \rho_{n-1}(x) |f(x)|^2 \ud x,
\label{eq:nm1_2_n_bound}
\end{equation}
where the constants $C$, $\alpha$ are uniformly bounded in $n$ and $N$.

\noindent (iii) 
The ratio $\frac{\rho_n}{\rho_{n-1}}\in H^1(\R^d;\rho_{n-1})$.

\noindent (iv) 
There exists $\bar{\lambda}>0$, such that $\rho_n$ satisfies {\bf PI}($\bar{\lambda}$) for each $n$.

\qed
\end{proposition}

\medskip

The sequence of minimizers $\{\rho_n\}$ is used to construct, via a
piecewise-constant interpolation, a density
function $\rho^{(N)}(x,t)$ for $t\in[0,T]$: 
Define $\rho^{(N)}(x,t)$ by setting $\rho^{(N)}(x,t_n) = \rho_{n}(x)$, and
taking $\rho^{(N)}$ to be constant on each time interval $[t_{n-1}, t_n)$
for $n=1,2,\hdots,N$.  

The following section is concerned with convergence analysis for the limit, as
$\barDelta_N \rightarrow 0$.  Before describing the analysis, we
present a few preliminaries concerning a certain Poisson's equation.  
This equation is fundamental to both the nonlinear filter
(in~\Sec{sec:gradient_flow}) and the particle
filter algorithm (in~\Sec{sec:FPF}).

\subsection{Poisson's Equation}  

We are interested in obtaining a solution $\phi$ of Poisson's equation,
\begin{equation}
\begin{aligned}
\nabla \cdot (\rho(x) \nabla \phi (x) ) & = - (g(x)-\hat{g})\rho(x),
\\
\int \phi(x) \rho(x) \ud x & = 0,
\label{eqn:EL_phi}
\end{aligned}
\end{equation}
where $\rho>0$ is a given density, $g$ is a given function, and 
$\hat{g} = \int g(x) \rho(x) \ud x$.

The terminology is motivated by Poisson's equation that arises in the theory of Markov processes \cite{glymey96a,MT}.  Consider the normalized \textit{Smoluchowski equation}, defined as the perturbed gradient flow w.r.t.\ a potential $U\colon\R^d\to\R^d$:
\[
\ud \Phi_t =-\nabla U(\Phi_t)\, \ud t  + \sqrt{2} \ud W_t.
\]
Its differential generator is the second-order operator, defined for $C^2$ functions by $\gen \phi = -(\nabla U)\cdot \nabla \phi +   \triangle \phi$.  On taking $U=-\ln(\rho)$,  the first equation in \eqref{eqn:EL_phi} becomes the usual Poisson's equation for diffusions,
\[
\gen\phi\, (x) =
- (g(x)-\hat{g}).
\]  
This interpretation is appealing, but will not be needed in subsequent analysis.  We henceforth consider solutions to \eqref{eqn:EL_phi} in a purely analytical setting.

Let $H_0^1(\R^d;\rho) \doteq \{ \phi \in H^1(\R^d;\rho) \,\Big{\vert} \, \int \phi(x) \rho(x) \ud x = 0 \}$. 
A function $\phi \in H_0^1(\R^d;\rho)$ is said to be a weak solution
of Poisson's equation~\eqref{eqn:EL_phi} if
\begin{equation}
\int \nabla \phi (x) \cdot \nabla \psi(x) \rho(x) \ud x = \int  (g(x)-\hat{g}) \psi(x) \rho(x) \ud x,\label{eqn:EL_phi_weak}
\end{equation}
for all $\psi \in H^1(\R^d;\rho)$.

The existence-uniqueness result for the weak solution of Poisson's equation is
described next;  its proof is given in the Appendix~\ref{Appdx:Poisson}.

\begin{theorem}\label{thm:thm1}
Suppose $\rho(x)=e^{-\smopot(x)}$ satisfies {\bf PI}($\lambda$).  

\noindent (i) If $g\in L^2(\R^d;\rho)$, then there exists a unique weak solution
$\phi \in H_0^1(\R^d;\rho)$ satisfying
\eqref{eqn:EL_phi_weak}. Moreover, the derivatives of the solution are
controlled by the size of the data:
\begin{equation}
\int |\nabla\phi|^2 \rho(x) \ud x  \le \frac{1}{\lambda} \int |g -
\hat{g}|^2 \rho(x) \ud x.\label{eqn:bound1}
\end{equation}
\noindent (ii) If $g\in H^1(\R^d;\rho)$ and $D^2 \smopot \in L^\infty$,
then the weak solution has higher regularity: $\phi \in H^2(\R^d;\rho)$ with
\begin{equation}
\int  \left| D^2 \phi \right|^2  \rho(x) \ud x  \le
C(\lambda;\rho) \int |\nabla g|^2 \rho(x) \ud x,
\label{eqn:bound2}
\end{equation}
where $C(\lambda;\rho)= \lambda^{-2} \big( \lambda +  \lVert D^2(\smopot)
    \rVert_{L^\infty} \big)$.
 \qed\end{theorem}

\medskip

\section{Nonlinear Filter} 
\label{sec:gradient_flow}

The analysis proceeds by first obtaining the first variation as
described in the following Lemma.  The proof appears in the Appendix~\ref{proof:lemma:EL}.

\begin{lemma}[First-order optimality condition]\label{lem:EL}
Consider the minimization problem~\eqref{eqn:obj_fn} under
Assumptions~(A1)-(A2).  
The minimizer $\rho_n$ satisfies the Euler-Lagrange equation
\begin{align}
\int \rho_{n} \big[ - \nabla \smopot_{n} \cdot \varsigma+ \nabla \smopot_{n-1} 
\cdot \varsigma - (\Delta Z_{n} - h \Delta t_n) \nabla h \cdot \varsigma \big] \ud x = 0
\label{eqn:EL-density}
\end{align}
for each vector field $\varsigma \in L^2(\R^d \to \R^d;\rho_{n-1})$.
\qed
\end{lemma}

\medskip

We are now prepared to state the main theorem concerning
the limit of the sequence of densities $\{\rho^{(N)}(x,t)\}$.
For the purpose of the proof, an alternate form of the
E-L equation is more useful.  For a given function 
$g\in L^2(\R^d;\rho_{n-1})$,
\notes{Reviewer wanted this notation.  Is it ok with you?  -spm}
let $\varsigma \in L^2(\R^d\to \R^d;\rho_{n-1})$ denote the weak solution (in gradient form) of
\begin{equation}
\nabla \cdot \left( \rho_{n-1}(x) \varsigma (x) \right) = - \left( g(x)
- \int \rho_{n-1}(x) g(x) \ud x \right)\rho_{n-1}(x).
\label{eq:varsigmadefn}
\end{equation}
Such a solution exists by Theorem~\ref{thm:thm1}~(i).  The E-L
equation~\eqref{eqn:EL-density} can then be expressed as
\begin{equation}
\int \rho_n(x) g(x) \ud x =  \int \rho_{n-1}(x) g(x) \ud x + \int
\rho_{n}(x) \left[ \Delta Z_{n} - h(x) \Delta t_n \right] \nabla h(x) \cdot
\varsigma (x) \ud x.
\label{eqn:el_appdx_pf1}
\end{equation}
The derivation of~\eqref{eqn:el_appdx_pf1}
from~\eqref{eqn:EL-density}-\eqref{eq:varsigmadefn} appears in
Appendix~\ref{appdx:derivationof-eqn:el_appdx_pf1}.  

Let us suppose now $\Delta t_n \rightarrow 0$ uniformly, so that
$\barDelta_N\to 0$ as $N\to\infty$,  where the maximum step size
$\barDelta_N$ was introduced in \eqref{e:DeltaN}.  
Based on the proof of Prop.~\ref{prop:exist_uniq}, there exists a
limit, denoted as $\rho(x,t)$, such that $\rho^{(N)}(x,t) \rightarrow
\rho(x,t)$ pointwise for a fixed sample path, and in the $L^2$ sense
over all sample paths. 
In fact for the special case of the signal
process~\eqref{eqn:Signal_Process} considered in this paper, the
limiting density is given by the following explicit formula:
\begin{equation}
\rho(x,t) \doteq (\text{const.}) \exp \left( h(x) \, (Z_t-Z_0) - \frac{1}{2}
|h(x)|^2 \, t \right) \rho_0(x).
\label{eqn:rho_limit_explicit}
\end{equation}
The convergence argument appears in Appendix~\ref{Appdx:Convergence}. 
   
The proof of the following theorem appears in
Appendix~\ref{proof:thm:KS}.  Notationally, $\langle f,\rho_t\rangle  \doteq \int
f(x) \rho(x,t) \ud x$ and $\hat{h}_t \doteq \int h(x) \rho(x,t) \ud x$.

\begin{theorem}\label{thm:KS}
The density $\rho$ is a weak solution of the nonlinear filter with prior $\rho_0=p_0^*$.  That is, for any test function $f \in C_c (\R^d)$,
\begin{equation}
\langle f,\rho_t\rangle  = \langle f,\rho_0\rangle  
+ \int_0^t 
		\langle(h - \hat{h}_s) (\ud Z_s -\hat{h}_s \ud s) f, \rho_s \rangle.
\label{eqn:KS}
\end{equation}
\qed
\end{theorem}

\begin{remark}
\upshape
\label{r:spectralBdds}
The considerations of this section highlight the variational
underpinnings of the nonlinear filter for the special case, $\ud X_t = 0$.  

For a general class of diffusions, the time-stepping procedure is
modified as follows: Set $\rho_0 = p_0^* \in {\cal P}$  and
inductively define $\{\rho_n\}_{n=1}^{N} \subset {\cal P}$ by taking
$\rho_n \in {\cal P}$ to minimize the functional \eqref{eqn:obj_fn},
\begin{equation*}
I_n(\rho)\doteq D(\rho \mid {\mathbb P} [{\rho}_{n-1}]) + \frac{\Delta t_n}{2} \int
\rho(x) (Y_n - h(x))^2 \ud x,
\end{equation*}
where ${\mathbb P} [\rho_{n-1}]$ is the ``push-forward'' from time $t_{n-1}$
to $t_n$, i.e., ${\mathbb P} [\rho_{n-1}]$ is the probability density
of $X_{t_n}$, given $\rho_{n-1}$ as the (initial) density of $X_{t_{n-1}}$.  For
the special case considered in this section, ${\mathbb P} [\rho_{n-1}]
= {\rho}_{n-1}$.  
\notes{I hate this.  Why can't we use standard notation,  ${\rho}_{n-1} P^{t_n-t_{n-1}}$, where $\{P^t\}$ is the semigroup for the diffusion?
In any case, we do need a subscript ${\mathbb P}_n$ since the increments of $t_n$ are not constant.}

The proof procedure is easily modified to derive the counterpart of the E-L
equation~\eqref{eqn:EL-density} and the nonlinear
filter~\eqref{eqn:KS} for a general class of diffusions.  
The hard part is to establish, in an a priori manner, the spectral
bound {\bf PI}($\bar{\lambda}$) in Prop.~\ref{prop:exist_uniq}.   Derivation of the spectal bound for the general case will
be a subject of future work.   Note that
the bound is needed to obtain a unique solution of the Poisson
equation.  

The following section shows that both the
variational analysis and the Poisson equation are also central to
construction of a particle filter algorithm in continuous time.     
\qed
\end{remark}

\section{Feedback Particle Filter}
\label{sec:FPF}

The objective of this section is to employ the time-stepping procedure
to construct a particle filter algorithm.  

\notes{ (archives) 
{\bf comprise} 
Usage: The use of of after comprise should be avoided: the library comprises (not comprises of) 6500,000 books and manuscripts. Consist, however, should be followed by of when used in this way: Her crew consisted of children from Devon and Cornwall.
\\
I believe that "comprised of" is o.k., and the web agrees: comprise is increasingly used in place of compose, especially in the passive: The Union is comprised of 50 states
}

A particle filter
is comprised of $N$ stochastic processes $\{X^i_t : 1\le i \le N\}$: The
value $X^i_t \in \R^d$ is the state for the $i^{\text{th}}$ particle
at time $t$. For each time $t$, the empirical distribution formed by
the ``particle population'' is used to approximate the posterior
distribution.  This is  defined for any measurable set $A\subset\R^d$
by
\begin{equation}
p^{(N)}(A,t) = \frac{1}{N}\sum_{i=1}^N \ind\{ X^i_t\in A\}. 
\label{e:piN}
\end{equation}
The model for the particle filter is assumed here to be a controlled system,
\begin{equation}
\ud X^i_t = \underbrace{u(X^i_t, t) \ud t + \v(X^i_t,t) \ud Z_t}_{\ud U_t^i},
 \label{eqn:particle_model}
\end{equation}  
where the functions $\v(x,t),\,u(x,t)$ are $\R^d$-valued.  It is assumed that  the initial conditions $\{X^i_0\}_{i=1}^N$  are
i.i.d., independent of $\{X_t,Z_t\}$, and drawn from the initial
distribution $p^*(x,0) \equiv p^*_0(x)$ of $X_0$.

We impose the following admissibility requirements on
the control input $U^i_t$ in~\eqref{eqn:particle_model}:
\begin{definition}[Admissible Input]
The control input $U^i_t$ is {\em admissible} if the following conditions are met:
(i)
The  random variables $u(x,t) $
and $\v(x,t)$ are  $\clZ_t = \sigma(Z_s:s\le t)$ measurable for each
$t$.   
(ii)
For each $i$ and $t$, $ \Expect[|u|] \doteq \Expect[ \sum_{l} | u_l(X^i_t,t) | ]<\infty$,
 and $\Expect[|\v|^2] \doteq \Expect [ \sum_j |\v_j(X^i_t,t) |^2 ]<\infty$.
\qed
\end{definition}

\medskip

There are two types of conditional distributions of interest in our analysis:
\begin{romannum}
\item $p(x,t)$:  Defines the conditional distribution of $X^i_t$
    given $\clZ_t$.
\item $p^*(x,t)$: Defines the conditional distribution of $X_t$
    given $\clZ_t$.
\end{romannum}

\medskip

The functions $\{ u(x,t),\v(x,t)\}$ are said to be \textit{optimal} if
$p\equiv p^*$.  That is, given $p^*(\cdot,0)= p(\cdot,0)$, our goal is
to choose $\{u,\v\}$ in the feedback particle filter so that the
evolution equations of these conditional distributions coincide.

\medskip

The optimal functions are obtained from the time-stepping procedure
introduced in~\Sec{sec:TS_prcdr}.  Recall that at step $n$ of the
procedure, the distribution $\rho_{n}$ is obtained upon minimizing
the functional~\eqref{eqn:obj_fn}, repeated below:
\begin{equation*}
I_n(\rho)\doteq D(\rho \mid \rho_{n-1}) + \frac{\Delta t_n}{2} \int
\rho(x) (Y_n - h(x))^2 \ud x.
\label{eqn:obj_fn_appdx0}
\end{equation*}
The optimizer has an explicit representation given in \eqref{eq:Bayes_with_Yn}.  
\notes{ spm archive,
too bad this equation gives rhon and not rhon+1}

The key is to construct a diffeomorphism $x\mapsto
s_n(x)$ such that $\rho = s_n^{\#} \left(\rho_{n-1} \right)$, where
$s_n^{\#}$ denotes the push-forward operator.  The push-forward of a probability density $\rho$ by a smooth map $s$ is defined through the change-of-variables formula
\begin{equation*}
	\int g(x) [s^{\#}(\rho)](x) \ud x = \int g(s(x)) \rho(x) \ud x,
\end{equation*}
for all continuous and bounded test functions $g$.

The particle filter
equations are obtained from the first-order
optimality conditions for $s_n$.  For this purpose, we look at the
cumulative objective function, defined for $N\ge 1$ by
\begin{equation}
J^{(N)}(\underline{s}) \doteq \sum_{n=1}^N \left(  I_n(s_n^\#(\rho_{n-1})) - \frac{\Delta t_n}{2} Y_n^2 \right),
\label{e:JN_main}
\end{equation}
where $\underline{s}\doteq (s_1,s_2,\hdots,s_N)$ denotes a sequence of
diffeomorphisms.  The objective is to construct a minimizer, denoted as
$\underline{\chi} \doteq (\chi_1,\chi_2,\hdots,\chi_N)$,
and consider the limit as
$N\rightarrow \infty$,  $\barDelta_N \rightarrow 0$.  Note the sequence
$\{ \rho_{n-1}(x) \}_{n=1}^N$ is assumed given here (see~\eqref{eq:Bayes_with_Yn}).
Its limit, which we denote as $\rho(x,t)$, see~\eqref{eqn:rho_limit_explicit}, is equal to $p^*(x,t)$, the posterior distribution of $X_t$
    given $\clZ_t$, by Theorem~\ref{thm:KS}.

\medskip

The calculations in Appendix~\ref{appdx:FPF} provide the following
characterization of the optimal functions $\{u,\v\}$:
\begin{romannum}
\item The function $\v$ is a solution to
\begin{equation}
\nabla \cdot (\rho \v) = - (h-\hat{h})
 \rho,
\label{e:bvp_divergence_multi}
\end{equation}
\item The function $u$ is obtained as
\begin{equation}
u(x,t) = -\frac{1}{2} \v(x,t) \bigl( h(x) + \hat{h}_t \bigr) + \w(x,t),
\label{eqn:u_intermsof_v*}
\end{equation}
where $\hat{h}_t \doteq \int h(x) \rho(x,t) \ud x$ and $\w =
\left(\w_1,\w_2,...,\w_d\right)$ is a $\R^d$-valued function with
\begin{equation*}
\w_l(x,t) := \frac{1}{2} \sum_{k=1}^d \v_{k}(x,t)
\frac{\partial \v_{l}}{\partial x_{k}}(x,t).
\label{eqn:wong_term_intro}
\end{equation*}
\end{romannum}
This in particular yields the following feedback particle filter
algorithm -- obtained upon substituting $\rho$ by $p$, the posterior
distribution of $X_t^i$ given
$\mathcal{Z}_t$:

\medskip

\noindent{\bf Feedback particle filter} (in Stratonovich
form) is given by
\begin{equation}
\begin{aligned}
\ud X^i_t 
&  =\v(X^i_t,t) \circ \ud \Inov^i_t, \end{aligned}
\label{eqn:particle_filter_nonlin_intro}
\end{equation}
where 
\begin{equation*}
\ud \Inov^i_t \doteq \ud Z_t - \frac{1}{2}    (h(X^i_t) + \hat{h}_t) \ud t\,, \qquad \hat{h}_t := \E[h(X_t^i)|\mathcal{Z}_t].
\label{e:in_intro}
\end{equation*}
The gain function is expressed as
\begin{equation*}
\v(x,t) = 
\nabla \phi(x,t),
\label{eqn:gradient_gain_fn_intro}
\end{equation*}
and it is obtained at each time $t$ as a solution of Poisson's equation:
\begin{equation*}
\begin{aligned}
\nabla \cdot (p(x,t) \nabla \phi (x,t) ) & = - (h(x)-\hat{h}) p(x,t),\\
\int \phi(x,t) p(x,t) \ud x & = 0,
\end{aligned}
\end{equation*}
where $p$ denotes the conditional distribution of $X_t^i$ given
$\mathcal{Z}_t$.
\qed

\medskip

This algorithm requires approximations in numerical
implementation since both the gain $\v$ and the conditional mean $\hat{h}$ depend upon the density $p$ to be estimated.
This is resolved by replacing $p$ by the empirical distribution \eqref{e:piN} to obtain $\hat{h}_t \approx
\frac{1}{N}\sum_{i=1}^N h(X_t^i) =: \hat{h}_t^{(N)}$.  Likewise, a
Galerkin algorithm is used to obtain a finite-dimensional
approximation of the gain function $\v$;
cf.,~\cite{yanlaumehmey12}.   

The following theorem shows that, in absence of these approximations,  the feedback particle filter is
exact.  Its proof appears in the Appendix~\ref{apdx:pf_Kushner}.

\begin{theorem}\label{thm:exact}
Under Assumptions~(A1)-(A2), the feedback particle filter~\eqref{eqn:particle_filter_nonlin_intro} is
exact.  That is, 
provided $p(\varble,0)=p^*(\varble ,0)$, we have for all $t\ge 0$,
\[
p(\varble,t) = p^*(\varble,t).
\]
\qed
\end{theorem}

\begin{remark}
\upshape
The extension of the feedback particle filter to the general nonlinear
filtering problem is straightforward.  In particular, consider the
filtering problem
\begin{align*}
\ud X_t &= a(X_t)\ud t + \ud B_t,
\\
\ud Z_t &= h(X_t)\ud t + \ud W_t,
\end{align*}
where $X_t\in\R^d$ is the state at time $t$, $Z_t \in\R$ is the
observation, $a(\varble)$, $h(\varble)$ are $C^1$
functions, and $\{B_t\}$, $\{W_t\}$ are mutually independent
standard Wiener processes.

For the solution to this problem, the feedback particle filter is
given by
\begin{equation*}
\ud X^i_t   =a(X_t^i)\ud t + \ud B_t^i + \v(X^i_t,t) \circ \ud \Inov^i_t, 
\label{eqn:particle_filter_nonlin1}
\end{equation*}
where the formulae for $\v$ and $\Inov^i$ are as before.  The
extension of the Theorem~\ref{thm:exact} to this more general case
requires a well-posedness analysis of the solution of Poisson's
equation.  The key is to obtain a priori spectral bounds (see also
Remark~\ref{r:spectralBdds}) which will be a
subject of future publication.
\qed
\end{remark}

\section{Appendix}
\label{sec:appdx}

The convergence proofs here require bounds in the almost-sure and $L^2$ senses.

Recall that we consider a finite time interval $[0,T]$,  and for each $N$ we consider a
discrete-time sequence $\{0, t_1, t_2,\hdots,t_N\}$ with $0 \le t_1 \le \ldots \le t_N = T$, and denote
$\Delta t_n \doteq t_n-t_{n-1}$.  We let $\barDelta_N= \max_n \Delta
t_n$, which is assumed to vanish as $N\to\infty$.

We use $C>0$ to denote a constant that may depend on $N$ and on the process path $\{Z_t\}$, but is uniformly bounded in $L^2$. Recall that the densities $\rho_0,\ldots,\rho_N$ are random objects that depend on the samples $Z_0,\ldots,Z_N$. In particular,  the observation process has continuous sample paths, so there exists such a $C$ for which  
$|Z_t| \le C$ for all $t\in[0,T]$.   

\subsection{Proof of Prop.~\ref{prop:exist_uniq}}
\label{proof:prop:exist_uniq}

\noindent (i) Using~\eqref{eq:Bayes_with_Yn}, $\rho_n(x) = c_n \exp \left( -\sum_{k=1}^n\phi_k(x) \right)\rho_0(x)$, where $c_n$ is
a normalizing constant and $\phi_k(x)=\frac{\Delta t_k}{2} (Y_k - h(x))^2$.  Therefore,
\begin{align*}
\smopot_n(x) & \doteq - \ln\, \rho_n(x)
= \smopot_{0} (x) + \sum_{k=1}^n \frac{\Delta
  t_k}{2} (Y_k - h(x))^2 - \ln(c_n).
\end{align*} 
Differentiating,
\begin{align*}
\nabla \smopot_n(x) & = \nabla \smopot_0(x) -  \sum_{k=1}^n \Delta t_k
(Y_k-h(x))\nabla h(x) \\
& = \nabla \smopot_0(x) -  (Z_{t_n}-Z_{t_0}) \nabla h(x) + t_n  h(x) \nabla
h(x),
\end{align*}
and similarly,
\[
\frac{\partial ^2 \smopot_n }{\partial x_i \partial x_j}
=
\frac{\partial ^2 \smopot_0}{\partial x_i \partial x_j} 
-
(Z_{t_n}-Z_{t_0}) \frac{\partial ^2 h}{\partial x_i \partial x_j} 
+ 
t_n
\left(\frac{\partial h}{\partial x_i}\frac{\partial h}{\partial x_j} +
  h \frac{\partial ^2 h}{\partial x_i \partial x_j} \right).
\]

From the assumption~(A2) on $h$, it follows that, if $\smopot_0$
satisfies the properties listed in assumption~(A1), then so does
$\smopot_n$.  This is because the sample paths of $Z_t$ are
a.s.\ continuous and thus bounded on $[0,T]$.  
\medskip

\noindent (ii) Using~\eqref{eq:Bayes_with_Yn},
\begin{equation}
\rho_n(x) = \rho_{n-1} (x)\;  \frac{\exp(- \frac{\Delta t_n}{2} Y_n^2 ) \; \exp\left( h(x) \Delta Z_n - \frac{\Delta t_n}{2}
|h(x)|^2 \right) }{\exp(- \frac{\Delta t_n}{2} Y_n^2 )
\; \int \rho_{n-1}
  (y) \exp\left( h(y) \Delta Z_n - \frac{\Delta t_n}{2}
|h(y)|^2 \right)  \ud y}.
\label{eq:Rick-made-me-write-this}
\end{equation}
On canceling the common term $\exp(- \frac{\Delta t_n}{2} Y_n^2 )$ from both
the numerator and denominator, we can write $\rho_n(x) = \rho_n(x)\exp\big(H_n(x)\big)/\int \rho_{n-1}(y)\exp\big(H_n(y)\big)\ud y$, where we have defined $H_n(x) =  h(x) \Delta Z_n - \frac{\Delta t_n}{2}
|h(x)|^2$. Since $\rho_{n-1}$ is a probability density, we have
\begin{align*}
	\left\| \frac{\rho_n}{\rho_{n-1}} \right\|_\infty \le \exp\big(\text{osc}(H_n)\big), \qquad \text{osc}(H_n) \doteq \sup H_n - \inf H_n.
\end{align*}
Because $\frac{\Delta t_n}{2} |h(x)|^2 \ge 0$, $\sup H_n \le |\Delta Z_n| \| h \|_\infty$, whereas $\inf H_n \ge - |\Delta Z_n| \| h \|_\infty - \frac{\Delta t_n}{2} \| h \|^2_\infty$. Combining these estimates, we get the bound
\begin{equation}
\left\| \frac{\rho_n}{\rho_{n-1}}\right\|_\infty \le  \exp \left( 2 |\Delta Z_n| \|h\|_{\infty} +  \frac{\Delta t_n}{2}  \|h\|_{\infty}^2 \right).
\label{eq:ratio_ineq}
\end{equation}
It follows that
\begin{align*}
\|f\|^2_{L^2(\R^d;\rho_n)} 
& = \int \rho_n(x) |f(x)|^2 \ud x \\
& \le \exp \left( 2 |\Delta Z_n| \|h\|_{\infty} +  \frac{\Delta t_n}{2}
  \|h\|_{\infty}^2 \right) 
\int \rho_{n-1} (x) |f(x)|^2 \ud x.
\end{align*}
The second equation provides the bound~\eqref{eq:nm1_2_n_bound} in
part~(ii) of the proposition with $C= \exp( \frac{\barDelta_N}{2} \|h\|_{\infty}^2)$ and $\alpha = 2 \|h\|_{\infty}$.
\notes{division by 2 removed here as well --spm}

Based on this and the definition \eqref{e:clP}, we see that the minimizer $\rho_n\in{\cal P}$ if
$\rho_{n-1}\in{\cal P}$ (take $f(x) = x$ to establish a bounded second moment).  By induction,
$\rho_n\in{\cal P}$ if $\rho_{0}\in{\cal P}$. 
\medskip

\noindent (iii) Denoting the quantity on the right-hand side of \eqref{eq:ratio_ineq} by ${\cal E}$, we conclude that the ratio $\rho_n/\rho_{n-1} \in L^2(\R^d; \rho_{n-1})$, 
with
\[
\int \left( \frac{\rho_n}{\rho_{n-1}} \right)^2 \rho_{n-1} \ud x  = \int
\left( \frac{\rho_n}{\rho_{n-1}} \right) \rho_{n} \ud x \le  {\cal E}.
\]

By a direct calculation, 
\[
\nabla \left( \frac{\rho_n}{\rho_{n-1}} \right)  = ( -\nabla \smopot_n +
\nabla \smopot_{n-1}) \frac{\rho_n}{\rho_{n-1}}  = (\Delta Z_n - \Delta t_n h )\nabla h\;\frac{\rho_n}{\rho_{n-1}}.
\]
The gradient is in $L^2(\R^d\to\R^d;\rho_{n-1})$ because,
\notes{$L^2(\R^d\to\Re^d;\rho_{n-1})$ ok?  -spm}
\begin{align*}
\int  \left| \nabla \left( \frac{\rho_n}{\rho_{n-1}} \right) \right|^2 \rho_{n-1} \ud x 
& = \int (\Delta Z_n - \Delta t_n h )^2 |\nabla h|^2  \frac{\rho_n}{\rho_{n-1}} \rho_{n} \ud x 
 \\
& \le 2 (|\Delta Z_n|^2 + |\Delta t_n|^2 \|h\|_{\infty}^2) \| \nabla h\|_{\infty}^2 \; {\cal E}.
\end{align*}

\medskip

\noindent (iv) We claim that $\rho_n(x) = e^{-v_n(x)}\rho_0(x)$ where $v_n(x)$
is uniformly bounded. Then 
$\rho_n$ satisfies {\bf PI}($\lambda_n$) with
\begin{equation}
\lambda_n = \exp(-\text{osc}(v_n)) \lambda_0. 
\label{eqn:Holley}
\end{equation}
This is because, for any $f$ satisfying  $\int \rho_n f \ud x = 0$,
\notes{condition needed --- noted by reviewer --spm}
\begin{align*}
\int &|f|^2 \rho_n(x) \ud x  = \int |f|^2 e^{-v_n(x)} \rho_0(x) \ud x \le
 e^{- \inf v_n}  \int |f|^2 \rho_0(x) \ud x 
\\
 & \le e^{- \inf v_n} \frac{1}{\lambda_0}  \int |\nabla f|^2 \rho_0(x)
 \ud x 
\le e^{\sup v_n- \inf   v_n}  \frac{1}{\lambda_0} \int |\nabla f|^2 \rho_n(x) \ud x.
\end{align*}
A uniform bound on $v_n$ yields a uniform bound on $\lambda_n$. 

We now prove the claim that $v_n$ is uniformly bounded. 
Using~\eqref{eq:Rick-made-me-write-this} iteratively, 
we can write $\rho_n(x) = e^{-v_n(x)}\rho_0(x)$ with 
\[
v_n(x) = -\psi_n(x) + \ln\Bigl(\int \rho_0 \exp(\psi_n ) \ud x\Bigr),
\]
where $\psi_n(x) = (Z_{t_n}-Z_{t_0}) h(x) - \frac{t_n}{2} |h(x)|^2$.
It then follows that
\[
\text{osc}(v_n) \le 2 \| \psi_n\|_{\infty} \le C \|h\|_{\infty} + T \|h\|_{\infty}^2,
\] 
where $C$ depends upon the sample path $Z_t$ for $t\in[0,T]$ but is
independent of $N$.

Using~\eqref{eqn:Holley}, $\rho_n$ satisfies {\bf
  PI}($\bar{\lambda}$) with $\bar{\lambda} = \lambda_0 \exp \left( - (C
  \|h\|_{\infty} + T \|h\|_{\infty}^2) \right)$.

\subsection{Convergence of $\{\rho^{(N)}\}$}
\label{Appdx:Convergence}
Now we explain in what sense $\rho^{(N)}$ converges to $\rho$ as $N \to \infty$. Recalling formula \eqref{eqn:rho_limit_explicit}, we have $\rho(x,t) = e^{-v(x,t)}\rho_0(x)$ with 
\[
v(x,t) = -\psi(x,t) + \ln\Bigl(\int \rho_0 \exp(\psi ) \ud x\Bigr),
\]
where $\psi(x,t) = (Z_{t}-Z_{0}) h(x) - \frac{t}{2} |h(x)|^2$.
Define $v^{(N)}(x,t) = v_n(x)$ whenever $t \in [t_{n-1},t_n)$. Assuming the maximum step size $\barDelta_N \rightarrow 0$ as $N \rightarrow \infty$ , we deduce that 
\[
v^{(N)} - v \rightarrow 0
\]
uniformly with respect to $x \in \R^d, t \in [0,T]$, due to the boundedness of $h$ and (uniform) continuity of the sample path $t \mapsto Z_t$. Hence 
\[
\frac{\rho^{(N)}}{\rho} = \exp(v-v^{(N)}) \rightarrow 1
\]
uniformly with respect to $x$ and $t$. In particular, $\rho^{(N)} \rightarrow \rho$ pointwise. 

\subsection{Proof of Theorem~\ref{thm:thm1}}
\label{Appdx:Poisson}

A density $\rho$ is assumed to satisfy {\bf PI}($\lambda$): That
is, for all functions $\phi\in H_0^1(\R^d;\rho)$,
\begin{equation}
\int |\phi(x)|^2 \rho(x) \ud x \le \frac{1}{\lambda} \int
|\nabla\phi(x)|^2 \rho(x) \ud x.
\label{eqn:Poincare}
\end{equation}

Consider  the inner product
\[
\langle\phi,\psi\rangle \doteq \int \nabla \phi(x) \cdot \nabla
\psi(x) \; \rho(x) \ud x.
\]
On account of~\eqref{eqn:Poincare}, the norm defined by using the
inner product $\langle \cdot,\cdot\rangle$ is equivalent to the standard norm in
$H^1_0(\R^d;\rho)$.
\notes{  the reviewer wants "dominated" instead of equivalent.  This may go back to a comment I made last year:  "I added the zero since this is not a norm on $H^1$".   Thoughts?  -spm}
\medskip

\noindent (i) Consider the BVP in its weak form~\eqref{eqn:EL_phi_weak}.  The
  integral on the right hand side is a bounded linear functional on
  $H_0^1$, since
\begin{align*}
\Big|\int  (g(x)-\hat{g}) \psi(x) \rho(x) \ud x\Big|^2 
& \le 
\Big( \int
|g(x)-\hat{g}|^2 \rho (x) \ud x \Bigr)
\Bigl( \int  |\psi(x)|^2 \rho (x) \ud x \Bigr)
\\
& \le
k_g \int  |\nabla \psi(x)|^2 \rho (x) \ud x,
\end{align*}
where~\eqref{eqn:Poincare} is used to obtain the second inequality,  with $k_g = \lambda^{-1} \int
|g(x)-\hat{g}|^2 \rho (x) \ud x $.

It follows from the Hilbert-space form of the Riesz representation theorem that there exists a
unique $\phi \in H_0^1$ such that
\[
\langle\phi,\psi\rangle  =  \int  (g(x)-\hat{g}) \psi(x) \rho (x) \ud
x
\]
holds for all $\psi \in H_0^1(\R^d;\rho)$.  It trivially also holds for all 
constant functions ($\psi\equiv \text{const.}$).  Hence, it holds for
all $\psi \in H^1(\R^d;\rho)$ and $\phi$ is a weak solution
of the BVP, satisfying \eqref{eqn:EL_phi_weak}.

The estimate~\eqref{eqn:bound1} follows by substituting $\psi=\phi$ in~\eqref{eqn:EL_phi_weak} and using
Cauchy-Schwarz.

\medskip

\noindent (ii) For the estimate~\eqref{eqn:bound2}, we first establish the
  following bound:
\begin{equation} \label{eqn:h2bound}
   \int |D^2 \phi|^2 \rho \ud x \leq  \int \nabla \phi \cdot G \rho \ud x,
\end{equation}
where the vector function $G \in L^2(\R^d\to\R^d;\rho)$ is defined by
\[
G = D^2(\log \rho) \nabla \phi + \nabla g
\]
and where $|D^2 \phi|^2= \sum_{j,k} (\frac{\partial^2 \phi}{\partial x_j \partial x_k})^2$.

Since each entry of the Hessian matrix $D^2(\log \rho)$ is bounded and
$\nabla g \in L^2(\R^d\to\R^d;\rho)$, we have $G \in L^2(\R^d\to\R^d;\rho)$.
The elliptic regularity theory \cite[Section 6.3]{EVANS}
applied to the weak solution $\phi \in H^1(\R^d;\rho)$ says that $\phi
\in H^3_{\rm loc}(\R^d)$.
Hence the partial differential equation holds pointwise:   
\begin{equation} 
- \nabla \cdot (\rho \nabla \phi) = (g - \hat{g}) \rho .
\label{eqn:pointwise}
\end{equation}
Differentiating with respect to $x_k$ gives
\begin{align*}
 - \nabla \cdot \left(\rho \nabla \frac{\partial \phi}{\partial x_k}\right) - \nabla \left(\frac{\partial \log p}{\partial x_k}\right) \cdot (\rho\nabla \phi) - \frac{\partial \log \rho}{\partial x_k} \nabla \cdot (\rho \nabla \phi) = \frac{\partial g}{\partial x_k} \rho + (g-\hat{g}) \frac{\partial \log \rho}{\partial x_k} \rho .
\end{align*}
The final terms on the left and right sides cancel, by equation \eqref{eqn:pointwise}. Thus the preceding formula becomes
\begin{equation} \label{eqn:pointwisederiv}
- \nabla \cdot \left(\rho \nabla \frac{\partial \phi}{\partial x_k}\right) = G_k \rho ,
\end{equation}

Let $\beta(x) \geq 0$ be a smooth, compactly supported ``bump'' function, meaning $\beta(x)$ is radially decreasing with $\beta(0)=1$. Let $s>0$, and multiply \eqref{eqn:pointwisederiv} by $\beta(sx)^2 \frac{\partial \phi}{\partial x_k}$. Integrate by parts on the left side (noting the boundary terms vanish because $\beta$ has compact support) to obtain 
\begin{equation}
\int \nabla \left[\beta(sx)^2 \frac{\partial \phi}{\partial x_k}\right] \cdot \left(\nabla \frac{\partial \phi}{\partial x_k}\right) \rho \ud x = \int \beta(sx)^2 \frac{\partial \phi}{\partial x_k}   G_k \,  \rho    \ud x .
\label{e:LHSRHS}
\end{equation}

 The left side of \eqref{e:LHSRHS} can be expressed as
\begin{align*}
 \int \beta(sx)^2 \left|\nabla \frac{\partial \phi}{\partial x_k}\right|^2 \rho
\ud x 
+ 2s \int \frac{\partial \phi}{\partial x_k} \beta(sx) (\nabla \beta)(sx) \cdot \left(\nabla \frac{\partial \phi}{\partial x_k}\right) \rho \ud x .
\end{align*}
The second term is bounded by
\begin{align*}
& 2s \lVert \nabla \beta \rVert_{L^\infty(\R^d)} \int \left|\frac{\partial
  \phi}{\partial x_k}\right| \beta(sx) \left|\nabla \frac{\partial \phi}{\partial
  x_k}\right| \rho \ud x \\
&\qquad \leq s \lVert \nabla \beta \rVert_{\infty} \int \left[\left(\frac{\partial \phi}{\partial x_k}\right)^2 + \beta(sx)^2 \left|\nabla \frac{\partial \phi}{\partial x_k}\right|^2 \right] \rho \ud x
\end{align*}
and so the left side of \eqref{e:LHSRHS} is bounded from below by
\begin{align*}
& (1-s \lVert \nabla \beta \rVert_{L^\infty(\R^d)}) \int \beta(sx)^2
\left|\nabla \frac{\partial \phi}{\partial x_k}\right|^2 \rho \ud x 
-  s \lVert \nabla \beta \rVert_{\infty} \int \left(\frac{\partial
    \phi}{\partial x_k}\right)^2 \ud x .  
\end{align*}
The right hand side of \eqref{e:LHSRHS} tends to $ \int \frac{\partial \phi}{\partial x_k} G_k \rho \ud x$, as $s\to 0$, 
 by dominated convergence, and  since $\beta(x)$ is radially decreasing, with $\beta(0)=1$. 

Letting $s \to 0$ in  \eqref{e:LHSRHS},    we conclude from the monotone convergence theorem that
\[
\int \left|\nabla \frac{\partial \phi}{\partial x_k}\right|^2 \rho \ud x \leq \int \frac{\partial \phi}{\partial x_k} G_k \rho \ud x .
\]
Summing over $k$ establishes the bound~\eqref{eqn:h2bound}.  

Next we prove \eqref{eqn:bound2}. First, 
\begin{equation} \label{eqn:interim}
\int |\nabla \phi|^2 \rho \ud x \leq \lambda^{-1} \int |g-\hat{g}|^2
\rho \ud x \leq \lambda^{-2} \int |\nabla g|^2 \rho \ud x
\end{equation}
by \eqref{eqn:bound1} followed by \eqref{eqn:Poincare} applied to the function $g-\hat{g} \in H^1_0(\R^d;\rho)$. Second, by the definition of $G$, the $L^2$-triangle inequality, and \eqref{eqn:interim}, we show that
\begin{align}
\big( \int |G|^2 \rho \ud x\big)^{\! 1/2} 
& \leq \lVert D^2(\log \rho) \rVert_{\infty} \big(\int |\nabla \phi|^2 \rho \ud x\big)^{1/2} + \big(\int |\nabla g|^2 \rho \ud x\big)^{\! 1/2} \notag \\
& \leq \Big( \frac{\lVert D^2(\log \rho) \rVert_{\infty}}{\lambda} + 1 \Big) \big(\int |\nabla g|^2 \rho \ud x\big)^{\! 1/2} . \label{eqn:interim2}
\end{align}
Now we take \eqref{eqn:h2bound} and apply Cauchy--Schwarz, followed by \eqref{eqn:interim} and \eqref{eqn:interim2}, to find:
\begin{align*}
\int |D^2 \phi |^2 \rho \ud x
& \leq \big( \int |\nabla \phi|^2 \rho \ud x \big)^{\! 1/2} \ \big( \int |G|^2 \rho \ud x \big)^{\! 1/2} \\
& \leq \big( \lambda^{-2} \int |\nabla g|^2 \rho \ud x \big)^{\! 1/2} \left( \frac{\lVert D^2(\log \rho) \rVert_{\infty}}{\lambda} + 1 \right) \big(\int |\nabla g|^2 \rho \ud x\big)^{\! 1/2} \\
& = \lambda^{-2} \big( \lambda + \lVert D^2(\log \rho) \rVert_\infty \big) \int |\nabla g|^2 \rho \ud x ,
\end{align*}
which proves \eqref{eqn:bound2}.

\subsection{Proof of Lemma~\ref{lem:EL}}
\label{proof:lemma:EL}
We compute the first variation of the functional \eqref{eqn:obj_fn}, which we reproduce here for reference:
\begin{equation}
I_n(\rho) := 
\int \rho(x)
\ln \rho(x) \ud x - \int \rho(x)
\ln \rho_{n-1}(x) \ud x 
+ \int \rho(x) \frac{(\Delta Z_{n}-h(x)\Delta t_n)^2}{2\Delta t_n} \ud x. 
\label{eqn:I_e_rho}
\end{equation}

Following the methodology of~\cite{Jordan99thevariational}, a vector field $\varsigma$ is used
to generate the first variation;  we initially assume that $\varsigma\in C^1_c$.    
Let $\Phi_{\tau}(x)$ be the solution of 
\[
\frac{\ud }{\ud \tau} \Phi = \varsigma(\Phi),\quad\quad
\Phi_0(x)  = x.
\]
For small $\tau$, define $\rho_{\tau} = \Phi_{\tau}^{\#} \rho_{n}$ to
be the push-forward of the minimizer $\rho_n$.
We have 
\[
\det(\nabla \Phi_{\tau}(x)) \, \rho_{\tau}(\Phi_{\tau}(x)) = \rho_{n}(x)\,,
\]
and $i(\tau)=I_{n}(\rho_{\tau})$ has a minimum at $\tau=0$. \notes{spm:  Someone else please address the reviewers concern}

The three terms in the E-L equation~\eqref{eqn:EL-density} are obtained
by explicitly evaluating the derivative $\frac{\ud}{\ud \tau}
i(\tau)$, at $\tau=0$, of the three terms in~\eqref{eqn:I_e_rho}:

\medskip

\noindent (i) The first term is the negative entropy 
\begin{align*}
\int \rho_{\tau} (z)
\ln \rho_{\tau} (z) \ud z & =  \int \rho_{n}(x) \ln \left[ \rho_{\tau}
  (\Phi_{\tau}(x)) \right] \ud x \\ 
& = \int  \rho_{n}(x) \ln \left[
  \rho_{n} (x) \, (\det (\nabla \Phi_{\tau}(x)))^{-1} \right] \ud x. 
\end{align*}
Therefore,
\begin{align*}
\left. \frac{\ud}{\ud \tau} \int \rho_{\tau} (z)
\ln \rho_{\tau} (z) \ud z \right|_{\tau=0} & =  -\left. \int \rho_{n}(x)
\frac{\ud}{\ud \tau} \, \ln [ \det (\nabla \Phi_{\tau}(x))
] \right|_{\tau=0} \ud x  \\
& = - \int \rho_{n}(x) \nabla \cdot \varsigma(x) \ud x = - \int
\rho_{n}(x) \nabla \smopot_n(x) \cdot \varsigma(x) \ud x,
\end{align*}
where the final equality is obtained by using integration by parts.  
The interchange of the order of the
differentiation and the integration is justified because the
difference quotient
\[
\frac{1}{\tau} \left( \ln [ \det (\nabla \Phi_{\tau}(x))
] - \ln [ \det (\nabla \Phi_{0}(x))
] \right)
\]
converges uniformly to $\left. \frac{\ud}{\ud \tau} \det (\nabla
  \Phi_{\tau}(x)) \right|_{\tau=0} =  \nabla \cdot \varsigma(x)$. This is
because $\varsigma$ is assumed to have a compact support and $\Phi_{\tau}(x) =  \Phi_{0}(x)
= x$ outside this compact set. 

\medskip

\noindent (ii) The second term is given by
\[
\int \rho_{\tau} (z)
\smopot_{n-1}(z) \ud z = \int \rho_{n} (x)
\smopot_{n-1} (\Phi_{\tau}(x)) \ud x,  
\]
and
\begin{align*}
\left. \frac{\ud}{\ud \tau} \int \rho_{n} (x)
\smopot_{n-1} \, (\Phi_{\tau}(x)) \ud x \right|_{\tau=0} & =  \int
\rho_{n}(x)  \left. \frac{\ud}{\ud \tau} \, \smopot_{n-1}
  (\Phi_{\tau}(x)) \right|_{\tau=0}  
\ud x \\
& = \int \rho_{n}(x) \, \nabla  \smopot_{n-1}(x) \cdot \varsigma(x)
\ud x,
\end{align*}
which is justified again because $\varsigma$ has compact support. 

\medskip

\noindent (iii) For the third term, similarly,
\begin{align*}
\left. \frac{\ud}{\ud \tau} [ \cdots ] \right|_{\tau=0}  &= \int \rho_n(x)
\left. \frac{\ud}{\ud \tau} \, \frac{(\Delta Z_{n}-h( \Phi_{\tau}(x)
    )\Delta t_n)^2}{2\Delta t_n} \right|_{\tau=0} \ud x\\
& = - \int \rho_{n} \left[ \Delta Z_{n} -
h(x) \Delta t_n \right] \nabla h(x) \cdot \varsigma(x) \ud x.
\end{align*}

Extension of the E-L equation to an arbitrary vector field in
$L^2(\R^d\to\R^d;\rho_{n-1})$ requires a standard approximation argument.  Suppose
$\varsigma\in L^2(\R^d \to \R^d;\rho_{n-1})$.  Using \Prop{prop:exist_uniq}~(ii),
$\varsigma\in L^2(\R^d \to \R^d;\rho_{n})$.  It then suffices to
approximate $\varsigma$ by a sequence of smooth, compactly supported vector
fields, noting that $|\nabla \smopot_{k}|(x) =O(|x|)$ as $x\rightarrow
\infty$, and that $h,\nabla h$
are bounded by assumption~(A2).  Recall here that ${\cal P}$ is the
space of probability densities with finite second moment. 
\notes{Reviewer asked if definition of $\cal P$ was used.  Comments?  -spm}

\medskip

\begin{remark}
Although the proof given here stresses the
variational aspect, the Euler-Lagrange equation can be obtained
directly from manipulating the solution~\eqref{eq:Bayes_with_Yn}: Taking the logarithm of~\eqref{eq:Bayes_with_Yn} gives
\[
- \smopot_{n} + \smopot_{n-1} + \phi_n = \text{const.}
\]
and applying the gradient operator yields:
\begin{equation*}
- \nabla \smopot_{n} + \nabla \smopot_{n-1} 
- (\Delta Z_{n} - h \Delta t_n) \nabla h = 0.
\end{equation*}
Multiplying by $\rho_n \varsigma$ and integrating gives~\eqref{eqn:EL-density}. 
\qed
\end{remark}
  
\subsection{Derivation of~\eqref{eqn:el_appdx_pf1}}
\label{appdx:derivationof-eqn:el_appdx_pf1}

Suppose $\varsigma$ is a weak solution of~\eqref{eq:varsigmadefn}.
Then for any test function $\psi\in H^1(\R^d;\rho_{n-1})$,
\begin{equation}
\int \nabla \psi(x) \cdot \varsigma (x) \; \rho_{n-1}(x) \ud x = \int
g(x) \psi(x) \; \rho_{n-1}(x) \ud x - \int
g \rho_{n-1} \ud x \int
\psi \rho_{n-1} \ud x .
\label{eq:weakformvarsigma}
\end{equation}

Take $\psi(x) = \frac{\rho_n(x)}{\rho_{n-1}(x)}$ -- the ratio is known to
be an element of $H^1(\R^d;\rho_{n-1})$ by
\Prop{prop:exist_uniq}~(iii).  The gradient of the ratio
is obtained as
\[
\nabla \left( \frac{\rho_n}{\rho_{n-1}} \right)  = ( -\nabla \smopot_n +
\nabla \smopot_{n-1}) \frac{\rho_n}{\rho_{n-1}}.
\]      
Substituting this in~\eqref{eq:weakformvarsigma}, 
\begin{equation}
\int  ( -\nabla \smopot_n +
\nabla \smopot_{n-1}) \cdot \varsigma (x) \rho_{n}(x) \ud x = \int
g(x) \rho_{n}(x) \ud x - \int
g(x) \rho_{n-1}(x) \ud x.
\label{eq:weakformvarsigma1}
\end{equation}

Combining~\eqref{eq:weakformvarsigma1} with~\eqref{eqn:EL-density}
gives the equation~\eqref{eqn:el_appdx_pf1}.

\subsection{Proof of Theorem~\ref{thm:KS}}
\label{proof:thm:KS}
We are given a test function $f\in C_c$.  So, $f\in
L^2(\R^d;\rho_n)$ for all $n\in\{1,2,\hdots,N\}$.  
Furthermore, there exists a uniform bound,
\begin{equation}
\|f\|_{L^2(\R^d;\rho_n)} < \|f\|_{L^\infty}  < C \quad\forall\;n.
\label{eq:bound_on_f}
\end{equation}
Denote $\hat{f}_n \doteq \int \rho_n(x) f(x)\ud x$.

\smallskip

Let $\xi_n \in L^2(\R^d\to\R^d;\rho_{n-1})$ be the weak solution of 
\begin{equation}
\nabla \cdot \left( \rho_{n-1}(x) \xi_n (x) \right) = - \left( f(x)
- \hat{f}_{n-1} \right)\rho_{n-1}(x).
\label{eqn:defn_xi}
\end{equation} 
Such a solution exists by Theorem~\ref{thm:thm1}, and moreover,
\begin{equation}
\int \rho_{n-1} |\xi_n|^2 \ud x < (\text{const.}) \int \rho_{n-1}
|f-\hat{f}_{n-1}|^2 \ud x < C,
\label{eq:bound_on_xi_n_min_1}
\end{equation}
where the (const.) is independent of $n$ (by
~\Prop{prop:exist_uniq}~(iv)), and using~\Prop{prop:exist_uniq}~(ii),
\begin{equation}
\int \rho_n(x) |\xi_n(x)|^2 \ud x \le C \exp(\alpha |\Delta Z_n|) \int \rho_{n-1}(x) |\xi_n(x)|^2 \ud x.
\label{eq:bound_on_xi}
\end{equation}

\smallskip

Using the E-L equation~\eqref{eqn:el_appdx_pf1} with $g=f$ and $\varsigma=\xi_n$, for $n=1,2,\hdots,N$:
\[
\hat{f}_n - \hat{f}_{n-1} = \int \rho_{n}(x) \left[ \Delta Z_{n} -
h(x) \Delta t_n \right] \nabla h(x) \cdot \xi_n(x) \ud x\,,
\]
and, upon summing,
\begin{equation}
\hat{f}_N = \hat{f}_0 + \sum_{n=1}^N \int \rho_{n}(x) \left[ \Delta Z_{n} -
h(x) \Delta t_n \right] \nabla h(x) \cdot \xi_n(x) \ud x.
\label{eqn:series}
\end{equation}

The remainder of the proof thus is to show that, as $\Delta t_n
\rightarrow 0$, the summation converges to the It\^o integral in~\eqref{eqn:KS},
where the convergence is in $L^2$. 

\medskip

We fix $n$, and express the summand as  
\begin{equation}
S_n:= \int \rho_n(x) \nabla h(x) \cdot \xi_n(x) \ud x
\, \Delta Z_n - \int \rho_n(x) h(x) \nabla h(x) \cdot \xi_n(x) \ud x \, \Delta t_n
\label{eqn:summand_defn}
\end{equation}
Each of these terms is well-defined because $\xi_n \in L^2(\R^d\to\R^d;\rho_{n})$ (see~\eqref{eq:bound_on_xi}), and $h,\nabla h\in L^\infty$. 

The two terms are simplified separately in the following two steps:

\smallskip

\noindent {\bf Step 1.} Consider the second term $- \left( \int \rho_n(x)
h(x) \nabla h(x) \cdot \xi_n(x) \ud x\right)  \Delta t_n$.  Let $\eta_n \in L^2(\R^d\to\R^d;\rho_{n-1})$
denote the weak solution of 
\[
\nabla \cdot \left( \rho_{n-1}(x) \eta_n (x) \right) = - \left( h(x) \nabla h(x) \cdot \xi_n(x)
- \int \rho_{n-1} h \nabla h \cdot \xi_n \ud x \right) \rho_{n-1}(x).
\] 
Repeating the earlier argument, using~\eqref{eq:bound_on_xi_n_min_1}
and the fact that $h,\nabla h\in L^\infty$,
\[
\int \rho_{n-1}(x)   |\eta_n(x)|^2 \ud x < C,
\]
and
\begin{equation}
\int \rho_n(x) |\eta_n(x)|^2 \ud x \le C \exp(\alpha |\Delta Z_n|) \int \rho_{n-1}(x) |\eta_n(x)|^2 \ud x.
\label{eq:bound_on_eta}
\end{equation}

\notes{Several errors were spotted by the reviewers in the text that follows.   I have corrected most.
I am confused about the "recurring error" involving $\Expect[|{\cal E}_n^{(1)}|^2]^\half$.  And, do we need to reduce $\alpha$ by two to deal with squares?  The reviewer thinks so.  I need to study this.  -spm}
\notes{yes, we need to divide by 2; and yes, the reviewer is also
  correct about the recurring error -- corrected - PGM}

Using the E-L equation~\eqref{eqn:el_appdx_pf1} with $g=h \nabla h
\cdot \xi_n$ and $\varsigma=\eta_n$,
\begin{equation}
\int \rho_n h \nabla h \cdot \xi_n \ud x \, \Delta t_n = \int \rho_{n-1}
h \nabla h \cdot \xi_n \ud x \, \Delta t_n + {\cal E}_n^{(1)},
\label{eqn:ea1}
\end{equation}
where 
\[
{\cal E}_n^{(1)} = \int \rho_n \nabla h \cdot \eta_n \ud x \, (\Delta Z_n \Delta
t_n) - \int \rho_n h \nabla h \cdot \eta_n \ud x \, (\Delta t_n)^2. 
\]
In order to establish convergence, we will require bounds for
the two integrals.  Since $h,\nabla h\in L^{\infty}$, using~\eqref{eq:bound_on_eta},
\begin{equation}
|{\cal E}_n^{(1)}| < C \exp(\frac{\alpha}{2} |\Delta Z_n|) \left( \int
  \rho_{n-1}(x) |\eta_n(x)|^2 \ud x \right)^{\frac{1}{2}} \left( |\Delta Z_n \Delta t_n| + (\Delta t_n)^2|\right).
\label{eq:bound_for_En1}
\end{equation}
Given the uniform $L^2$ bound on $C$, it follows that $\Expect[|{\cal E}_n^{(1)}|^2]^\half = O(\barDelta_N^{3/2})$, uniformly in $n$.

\smallskip

\noindent {\bf Step 2.}  The calculation for the first term in \eqref{eqn:summand_defn},
$\left(\int \rho_n(x) \nabla h(x) \cdot \xi_n(x) \ud x\right) \Delta Z_n$, is similar. Let
$\zeta_n\in L^2(\R^d\to\R^d;\rho_{n-1})$ denote the weak solution of 
\begin{equation}
\nabla \cdot \left( \rho_{n-1}(x) \zeta_n (x) \right) = - \left( \nabla h(x) \cdot \xi_n(x)
- \int \rho_{n-1} \nabla h \cdot \xi_n \ud x \right) \rho_{n-1}(x).
\label{eqn:defn_zeta}
\end{equation}
As before,
$\int \rho_{n-1} |\zeta_n|^2 \ud x < C$,   
and
using the E-L equation~\eqref{eqn:el_appdx_pf1} with $g=\nabla
h\cdot\xi_n$ and $\varsigma = \zeta_n$,
\begin{equation}
\int \rho_n \nabla h \cdot \xi_n \ud x \, \Delta Z_n = \int \rho_{n-1}
\nabla h \cdot \xi_n\ud x \, \Delta Z_n + \int \rho_n \nabla h \cdot \zeta_n \ud x \, (\Delta
Z_n)^2 + {\cal E}_n^{(2)},
\label{eqn:ea2}
\end{equation}
where 
\[
{\cal E}_n^{(2)} = - \int \rho_n h \nabla h \cdot \zeta_n \ud x \, (\Delta Z_n \Delta t_n), 
\]
and using the a priori bound for $\zeta_n$,
\begin{equation}
|{\cal E}_n^{(2)}| < C  \exp(\frac{\alpha}{2} |\Delta Z_n|) \left( \int \rho_{n-1}(x) |\zeta_n(x)|^2 \ud x \right)^{\frac{1}{2}}  \; |(\Delta Z_n \Delta t_n)|.
\label{eq:bound_for_En2}
\end{equation}

\smallskip

Using the two formulae~\eqref{eqn:ea1} and~\eqref{eqn:ea2}
from Steps 1 and 2, the summand~\eqref{eqn:summand_defn} is given by
\begin{align}
S_n & = \int \rho_n(x) \nabla h(x) \cdot \xi_n(x) \ud x
\, \Delta Z_n - \int \rho_n(x) h(x) \nabla h(x) \cdot \xi_n(x) \ud x
\, \Delta t_n
\nonumber\\
& = \int \rho_{n-1} \nabla h \cdot \xi_n \ud x \, \Delta Z_n + \int
\rho_n \nabla h \cdot \zeta_n \ud x \, (\Delta Z_n)^2 - \int \rho_{n-1}
h \nabla h \cdot \xi_n \ud x \, \Delta t_n \nonumber\\
& \quad +  {\cal E}_n^{(1)} +  {\cal E}_n^{(2)}.
\label{eqn:summand_step12}
\end{align}
Both error terms satisfy $\Expect[|{\cal E}_n^{(i)}|^2]^\half =
O(\barDelta_N^{3/2})$ for $i=1,2$.

In the following step, the first two integrals
in~\eqref{eqn:summand_step12} are further simplified.

\smallskip

\noindent {\bf Step 3.} For the first integral, integration by parts
gives
\begin{align}
\int \rho_{n-1} \nabla h \cdot \xi_n \ud x \, \Delta Z_n &= - \int h \nabla \cdot
(\rho_{n-1} \xi_n) \ud x \, \Delta Z_n \nonumber\\
& = \int \rho_{n-1}(x) h(x) (f(x) - \hat{f}_{n-1}) \ud x \, \Delta Z_n,
\label{eqn:summand_step12_term1}
\end{align}
where the second equality follows from~\eqref{eqn:defn_xi}.  

\smallskip

For simplifying the second integral, the E-L
equation~\eqref{eqn:el_appdx_pf1} is used once more.  As before,  let
$\varphi_n\in L^2(\R^d\to\R^d;\rho_{n-1})$ denote the weak solution of 
\[
\nabla \cdot \left( \rho_{n-1}(x)  \varphi_n(x)  \right) = - \left( \nabla
  h(x) \cdot \zeta_n(x)
- \int \rho_{n-1} \nabla h \cdot \zeta_n \ud x \right) \rho_{n-1}(x),
\] 
together with an a priori bound $\int \rho_{n-1} |\varphi_n|^2 \ud x < C$.

The E-L equation~\eqref{eqn:el_appdx_pf1} then gives
\begin{equation}
\int \rho_n \nabla h \cdot \zeta_n \ud x \, (\Delta Z_n)^2 = \int \rho_{n-1}
\nabla h \cdot \zeta_n \ud x \, (\Delta Z_n)^2 + {\cal E}_n^{(3)},
\label{eqn:summand_step12_term2}
\end{equation}
where 
\[
{\cal E}_n^{(3)} =  \int \rho_n \nabla h \cdot \varphi_n \ud x \,
(\Delta Z_n)^3 - \int \rho_n h \nabla h \cdot \varphi_n \ud x \,
(\Delta Z_n)^2 \Delta t_n,
\]
and using the a priori bound for $\varphi_n$,
\notes{zeta now varphi --spm}
\begin{equation}
|{\cal E}_n^{(3)}| < C  \exp(\frac{\alpha}{2}|\Delta Z_n|) \left( \int \rho_{n-1}(x) |\varphi_n(x)|^2 \ud x \right)^{\frac{1}{2}} \; \left( |(\Delta Z_n)^3| + |(\Delta Z_n)^2
  \Delta t_n)|\right).
\label{eq:bound_for_En3}
\end{equation}
Hence this third error term is also uniformly bounded,    $\Expect[|{\cal E}_n^{(3)}|^2]^\half = O(\barDelta_N^{3/2})$.

Substituting the
formulae~\eqref{eqn:summand_step12_term1}-\eqref{eqn:summand_step12_term2}
in~\eqref{eqn:summand_step12}, the summand is given by
\begin{align}
S_n  & =  \int \rho_{n-1} h (f - \hat{f}_{n-1}) \ud x \,
\Delta Z_n \nonumber\\
 & \quad + \int \rho_{n-1}
\nabla h \cdot \zeta_n \ud x \, (\Delta Z_n)^2 - \int \rho_{n-1}
h \nabla h \cdot \xi_n \ud x \, \Delta t_n \nonumber\\
& \quad +  {\cal E}_n^{(1)} +  {\cal E}_n^{(2)} +  {\cal E}_n^{(3)},
\label{eqn:summand_step123}
\end{align}
where recall $\xi_n$ is defined by~\eqref{eqn:defn_xi} and $\zeta_n$
by~\eqref{eqn:defn_zeta}.

Now, using integration by parts together with~\eqref{eqn:defn_xi} and~\eqref{eqn:defn_zeta},
\begin{align*}
\int \rho_{n-1}
\nabla h \cdot \zeta_n \ud x &= -\int h \nabla \cdot (\rho_{n-1}
\zeta_n) \ud x\nonumber\\
&= \int \rho_{n-1}h\nabla h\cdot \xi_n \ud x - \int \rho_{n-1} h \ud x
\, \int \rho_{n-1} \nabla h\cdot \xi_n \ud x\nonumber\\
&= \int \rho_{n-1}h\nabla h\cdot \xi_n \ud x +  \int \rho_{n-1} h \ud x
\, \int h \nabla  \cdot( \rho_{n-1} \xi_n) \ud x\nonumber\\
&= \int \rho_{n-1}h\nabla h\cdot \xi_n \ud x -  \int \rho_{n-1} h \ud x
\,\int \rho_{n-1} h (f - \hat{f}_{n-1}) \ud x.\nonumber
\end{align*}
Substituting the result of this calculation
in~\eqref{eqn:summand_step123}, the summand is given by
\begin{align}
S_n  & =  \int \rho_{n-1} h (f - \hat{f}_{n-1}) \ud x \,
\Delta Z_n \nonumber\\
 &  -  \int \rho_{n-1} h \ud x
\,\int \rho_{n-1} h (f - \hat{f}_{n-1}) \ud x \, (\Delta Z_n)^2 + \int \rho_{n-1}
h \nabla h \cdot \xi_n \ud x \,  ((\Delta Z_n)^2 - \Delta t_n) \nonumber\\
& \quad +  {\cal E}_n^{(1)} +  {\cal E}_n^{(2)} +  {\cal E}_n^{(3)}.
\label{eqn:summand_step123plus}
\end{align}

\smallskip
\smallskip

\noindent {\bf Step 4.}
Substituting the summand~\eqref{eqn:summand_step123plus} in the
series~\eqref{eqn:series} and letting $\Delta t_n\rightarrow 0$, we arrive at the It\^o integral:
\begin{align}
\hat{f}_t & = \hat{f}_0 + \int_0^t \int \rho(x,s) h (x) (f(x) -
\hat{f}_s) \ud x \, \ud Z_s \nonumber \\
 & \quad \quad \quad +  \int_0^t \int \rho(x,s) h (x) \ud x \,
\int \rho(x,s) h (x) (f(x) - \hat{f}_s) \ud x \, \ud s \nonumber\\
& = \hat{f}_0 + \int_0^t \int \rho(x,s) (h(x) - \hat{h}_s) f(x)
\ud x \, (\ud Z_s - \hat{h}_s
\ud s).
\end{align}

Convergence is obtained on applying the following $L^2$ limits.

\noindent (i) Since $\xi_n$ is a weak solution of the Poisson's
equation~\eqref{eqn:defn_xi}, 
\[
\sum_{n=1}^{N} \left( \int \rho_{n-1}
h \nabla h \cdot \xi_n \ud x \right) ((\Delta Z_n)^2 - \Delta t_n)
\rightarrow 0 \quad \text{as}\;\; N\rightarrow\infty.
\]
The proof of this limit is based on the following result for the summand.  Fix $s\in\R_+$, and let $n$ and $N$ tend to infinity in such a way that $t_n\to s$ as $n,N\to \infty$.  We then have
 \notes{reviewer says, careless!  Revised to make sure $t_n\to s$} 
\[
\begin{aligned}
\lim_{n,N\to\infty}
 \int \rho_{n-1} h \nabla h \cdot \xi_n \ud x 
 & = 
\lim_{n,N\to\infty} \frac{1}{2} \int h^2(f -\hat{f}_{n-1}) \rho_{n-1} \ud x  
\\
 &=
 \half \int h^2(f-\hat{f}_s) \rho \ud x.
\end{aligned}
\]

\medskip

\noindent (ii) The apriori bounds~\eqref{eq:bound_for_En1},~\eqref{eq:bound_for_En2}
and~\eqref{eq:bound_for_En3} are used to show that
\begin{align*}
{\cal E} =:
\sum_{n=1}^{N} |{\cal E}_n^{(1)}|  +  |{\cal E}_n^{(2)}| +  |{\cal
  E}_n^{(3)}| \rightarrow 0 \quad \text{as}\;\; N\rightarrow\infty\,, 
\end{align*}
where the convergence is in $L^2$. This follows because we have the
bound $\Expect[{\cal E}^2] = O(\barDelta^{1/2})$.  \notes{Reviewer asks, who cares?  --spm}

\subsection{Derivation of the feedback particle filter}
\label{appdx:FPF}

We consider the cumulative objective function~\eqref{e:JN_main}, repeated below:
\begin{equation}
J^{(N)}(\underline{s}) \doteq \sum_{n=1}^N \left(  I_n(s_n^\#(\rho_{n-1})) - \frac{\Delta t_n}{2} Y_n^2 \right),
\label{e:JN}
\end{equation}
where $\underline{s} \doteq (s_1,s_2,\hdots,s_N)$ denotes a sequence of
diffeomorphisms.  The sequence $\{ \rho_{n-1}(x) \}_{n=1}^N$ is
assumed given here (see~\eqref{eq:Bayes_with_Yn}).  The
objective is to construct a minimizer, denoted as $\underline{\chi} \doteq
(\chi_1,\chi_2,\hdots,\chi_N)$, and consider the limit as
$N\rightarrow \infty$,  $\barDelta_N \rightarrow 0$.  

The calculations in this section are strictly formal.  Generally, the
technicalities are downplayed in the interest of succinctly describing 
the main calculations. The Einstein tensor notation is employed for some
of the more laborious calculations.

The optimization problem~\eqref{e:JN} can be considered term-by-term
since $\{\rho_{n-1}\}$ is fixed for fixed $N$ and $\Delta t_n$.   With
these parameters fixed,  and attention focused to the $n$th summand,
we recast the optimization problem as one over $s_n$ as follows:
\notes{reviewer asks about existence and uniqueness --spm}
\begin{equation}
\begin{aligned}
\!\!
I_n(s_n) \doteq - \int \rho_{n-1} (x) &\ln(\det (D s_n (x))) \ud x  - \int \rho_{n-1} (x) \ln \frac{\rho_{n-1} (s_n(x))}{\rho_{n-1} (x)} \ud x\\
& + \frac{\Delta t_n}{2} \int \rho_{n-1} (x) (Y_n - h(s_n(x)))^2
\ud x,
\end{aligned}
\label{eqn:obj_y}
\end{equation}
where we have used the identity $\rho_n(s_n(x)) \det(D s_n (x))
=\rho_{n-1}(x)$.  As in the initial problem formulation, the minimizer
is denoted as $\chi_n$.  The minimal value exists because the
functional $I_n(\cdot)$ is bounded from below -- see the discussion
following the introduction of the functional $I_n(\rho)$
in~\eqref{eqn:obj_fn}.  In fact, a minimizer may be obtained in
closed form by considering the transport problem   
\[
s_n^\#(\rho_{n-1}) = \rho_n.
\]
Existence of solutions to such problems have been
extensively investigated in the optimal transportation literature;
cf.,~\cite{Villani_opt_transport}.  As with the derivation of the
nonlinear filter, we proceed via analysis of the first variation.
Such an approach is more tractable and leads to the elegant form of the feedback particle filter.  Once
the filter has been derived, its optimality is established by showing
the filter to be exact; cf.,~Proof of Theorem~\ref{thm:exact} in
Sec.~\ref{apdx:pf_Kushner}.     

The first-order conditions for optimization problem~\eqref{eqn:obj_y}
appear in the following Lemma.
Given $\pertu\in C_c^1(\R^d,\R^d)$, the directional derivative is denoted
\[
\delta I_n(\chi_n) \cdot \pertu \doteq \frac{\ud}{\ud \epsy} I_n(\chi_n +\epsy \pertu) \Big|_{\epsy=0} .
\]

\begin{lemma}[First-Order Optimality Conditions] 
\label{lem:EL2}
Consider the minimization problem~\eqref{eqn:obj_y} under
Assumptions~(A1)-(A2).  
The first-order optimality condition
 for the minimizer $\chi_n(x)$ is given by
\begin{equation}
\begin{aligned}
0 = \delta I_n(\chi_n) \cdot \pertu    =&
 \int \rho_{n-1}(x)
\tr \left( D \chi_n^{-1}(x) D \pertu (x) \right) \ud x  
\\
&
+
 \int \rho_{n-1}(x)
\frac{1}{\rho_{n-1}(\chi_n(x))} \nabla \rho_{n-1} (\chi_n(x)) 
\cdot \pertu (x) \ud x   
\\
&
+
  \int \rho_{n-1}(x) \left( \Delta Z_{n} -
h(\chi_n(x)) \Delta t_n \right) \nabla h(\chi_n(x)) \cdot \pertu (x) \ud x ,
\end{aligned}
\label{eqn:EL-transport}
\end{equation}
where $\pertu\in C_c^1(\R^d,\R^d)$ is an arbitrary
perturbation of $\chi_n$.
\qed
\end{lemma}
\begin{proof}
The three terms in~\eqref{eqn:EL-transport} are obtained by explicitly evaluating the derivative $\frac{\ud}{\ud \epsy} I_n(\chi_n +\epsy \pertu)$, at $\epsilon=0$, for the three terms in~\eqref{eqn:obj_y}:

\noindent (i) The first term is given by
\[
- \int \rho_{n-1} (x) \left[ \ln(\det (D\chi_n(x))) + \ln(\det(I+\epsilon
D\chi_n^{-1}(x) D\pertu (x))) \right] \ud x.
\]
Therefore, for the first term,
\begin{align*}
\left. \frac{\ud}{\ud \epsilon} [ \cdots ] \right|_{\epsilon=0} 
& = - \int
\rho_{n-1} (x) \left. \frac{\ud}{\ud \epsilon} \ln(\det(I+\epsilon
D\chi_n^{-1}(x) D\pertu (x))) \right|_{\epsilon=0} \ud x\\
& = -  \int
\rho_{n-1} (x) \text{tr} (D\chi_n^{-1}(x) D\pertu (x)) \ud x.
\end{align*}

\medskip

\noindent (ii) The second term is obtained by a direct calculation
\begin{align*}
\left. \frac{\ud}{\ud \epsilon} [ \cdots ] \right|_{\epsilon=0}  &=
- \int \rho_{n-1} (x) \left. \frac{\ud}{\ud \epsilon} \ln \rho_{n-1}
  (\chi_n(x) + \epsilon \pertu (x))  \right|_{\epsilon=0} \ud x \\
& =  - \int \rho_{n-1}(x)
\frac{1}{\rho_{n-1}(\chi_n(x))} \nabla \rho_{n-1} (\chi_n(x)) 
\cdot \pertu  (x) \ud x.
\end{align*}

\medskip

\noindent (iii) Similarly for the third term,
\begin{align*}
\left. \frac{\ud}{\ud \epsilon} [ \cdots ] \right|_{\epsilon=0}  &=
\frac{\Delta t_n}{2} \int \rho_{n-1} (x)
\left. \frac{\ud}{\ud \epsilon} \, (Y_{n}-h( \chi_n(x) + \epsilon
  \pertu (x))
    )^2 \right|_{\epsilon=0} \ud x\\
& = - \int \rho_{n-1}(x) \left( \Delta Z_{n} -
h(\chi_n(x)) \Delta t_n \right) \nabla h(\chi_n(x)) \cdot \pertu  (x) \ud x.
\end{align*}
\end{proof} 

Since our interest is in the limit as $\Delta t_n\rightarrow 0$ and
$N\to \infty$, we now restrict to  diffeomorphisms of the form
$\chi_n(x) = x + \v(x,n) \Delta Z_n + u(x,n) \Delta t_n$, where the
appropriate function spaces are: $\v\in H^1(\R^d\to\R^d;\rho_{n-1})$ and $u\in
H^1(\R^d\to\R^d;\rho_{n-1})$.  Starting from~\eqref{eqn:EL-transport}, the following is established in Appendix~\ref{appdx:EL_K_u_derivation}:
\begin{equation}
\begin{aligned}
\delta I_n(\chi_n) \cdot \pertu &  = 
E_z(n)  \, \Delta Z_n +
E_\Delta(n) \, \Delta t_n + O(\Delta t_n^2,\Delta Z_n \Delta
  t_n,\Delta Z_n^3),
\end{aligned}
\label{e:Jpertn}
\end{equation}
where, denoting $\v(x,n)\doteq (\v_1(x,n),\hdots,\v_d(x,n))$, 
$u(x,n)\doteq (u_1(x,n),\hdots,u_d(x,n))$ and expressing $\pertu(x)=\pertu(x,n)\doteq (\pertu_1(x),\hdots,\pertu_d(x))$, 
the following equations give expressions for $E_z$ and $E_\Delta$
(expressed using Einstein's tensor notation):   
\begin{eqnarray} 
E_z &=& - \int \frac{\partial}{\partial x_j} \left( \rho_{n-1} \frac{\partial
    \v_j}{\partial x_i} \right)\, \pertu_i \ud x  - \int \rho_{n-1} \frac{\partial^2 \ln \rho_{n-1}}{\partial
  x_i \partial x_j} \, \v_j \, \pertu_i \ud x  - \int \rho_{n-1} \frac{\partial h}{\partial x_i} \pertu_i 
  \ud x 
  \nonumber
  \\
& =& -\int \rho_{n-1} \frac{\partial }{\partial x_i} \left( \frac{1}{\rho_{n-1}}
  \frac{\partial}{\partial x_j} (\rho_{n-1} \v_j ) \right) \pertu_i 
\ud x - \int \rho_{n-1} \frac{\partial h}{\partial x_i} \pertu_i 
  \ud x,   
   \label{e:E2}
\\[.2cm]
E_\Delta & =& -\int \rho_{n-1} \frac{\partial }{\partial x_i} \left( \frac{1}{\rho_{n-1}}
  \frac{\partial}{\partial x_j} (\rho_{n-1} u_j ) \right) \pertu_i \ud x 
+ \int \rho_{n-1} \left( h \frac{\partial
    h}{\partial x_i} - \frac{\partial^2 h}{\partial
  x_i \partial x_j} \v_j \right)\, \pertu_i 
  \ud x 
       \nonumber
  \\
& & \qquad 
- \frac{1}{2} \int \rho_{n-1}
\frac{\partial^3 \ln \rho_{n-1}}{\partial
  x_i \partial x_j \partial x_k} \, \v_j\,  \v_k \, \pertu_i
(x) \ud x  
+ \int
\frac{\partial }{\partial x_j} \left( \rho_{n-1} \frac{\partial
    \v_j}{\partial x_k} \frac{\partial
    \v_k}{\partial x_i}  \right)\, \pertu_i \ud x. 
\nonumber\\
\label{e:E3}
\end{eqnarray}

We now return to the objective function $J^{(N)}(s) $ defined in
\eqref{e:JN}.  For any fixed $N$, the first order optimality condition
for the minimizer $\underline{\chi} \doteq
(\chi_1,\chi_2,\hdots,\chi_N)$ is now immediate:
\begin{equation}
\begin{aligned}
0=\delta J^{(N)}(\underline{\chi})\cdot\underline{\pertu}  =
\sum_{n=1}^N 
E_z(n)  &\Delta Z_n +
E_\Delta(n) \Delta t_n 
+
\sum_{n=1}^N
\Bigl[ O(\Delta t_n^2,\Delta Z_n \Delta
  t_n,\Delta Z_n^3) \Bigr],
\end{aligned}
\label{e:Jpert}
\end{equation}
where $\underline{\pertu}(x) \doteq (\pertu(x,1),\hdots,\pertu(x,N))$
and $\pertu(\cdot,n)\in C_c^1(\R^d,\R^d)$ is an arbitrary
perturbation.  Recall now, $\chi_n(x) \doteq x + \v(x,n) \Delta Z_n +
u(x,n)\Delta t_n$.  The sequence $\{\rho_n\}$, $\{\v(x,n)\}$, $\{u(x,n)\}$ and $\{\pertu(x,n)\}$ are
used to construct, via interpolation, $\rho^{N}(x,t)$, $\v^{(N)}(x,t)$,
$u^N(x,t)$ and $\pertu^N(x,t)$, respectively.  
Recall $\rho^{(N)} \rightarrow
\rho(x,t)$, given in~\eqref{eqn:rho_limit_explicit}.  Likewise we
formally denote the limit of  $\v^{(N)}(x,t)$,
$u^N(x,t)$ and $\pertu^N(x,t)$ as $\v(x,t)$, $u(x,t)$ and $\pertu(x,t)$,
respectively.

With this notation, the right-hand side
of~\eqref{e:Jpert}, as $N\rightarrow \infty$, is expressed as an It\^{o} integral,
\notes{missing subscript $s$ and $i$ here - I have not yet repaired.   I did make stylistic changes below.
I have stopped serious editing here since we need time to read. -spm}
\begin{align*}
& - \int_0^T \int \rho(x,s) \left( \frac{\partial }{\partial x_i} \Big( \frac{1}{\rho}
  \frac{\partial}{\partial x_j} (\rho \v_j ) \Big) + \frac{\partial
  h}{\partial x_i} \right) \pertu_i (x,s) \ud x \, \ud Z_s 
\\ 
& - \int_0^T \int \rho (x,s) \left( \frac{\partial }{\partial x_i} \Big( \frac{1}{\rho}
  \frac{\partial}{\partial x_j} (\rho u_j ) \Big) - h \frac{\partial
    h}{\partial x_i} + \frac{\partial^2 h}{\partial
  x_i \partial x_j} \v_j \right.
\\ &\quad\quad\quad\quad \left.
+ \frac{1}{2} 
\frac{\partial^3 \ln \rho}{\partial
  x_i \partial x_j \partial x_k} \, \v_j\,  \v_k  - \frac{1}{\rho} \frac{\partial }{\partial x_j} \Big( \rho \frac{\partial
    \v_j}{\partial x_k} \frac{\partial
    \v_k}{\partial x_i}  \Big) \right) \pertu_i (x,s) \ud x \, \ud s.
\end{align*}

Since $\delta J^{(N)}(\underline{\chi})\cdot \underline{\pertu}=0$ by optimality, and $\underline{\pertu}$ is arbitrary,  we obtain weak-sense differential equations for $\v$ and $u$.   The following two equations follow, also defined in the weak sense: 
\begin{align}
\frac{\partial }{\partial x_i} \left( \frac{1}{\rho}
  \frac{\partial}{\partial x_j} (\rho \v_j ) \right) & = - \frac{\partial
  h}{\partial x_i},
\label{eq:EL_Kappdx}
\\
\frac{\partial }{\partial x_i} \left( \frac{1}{\rho}
  \frac{\partial}{\partial x_j} (\rho u_j ) \right) & = h \frac{\partial
    h}{\partial x_i} - \frac{\partial^2 h}{\partial
  x_i \partial x_j} \v_j 
  \nonumber
  \\
  &
  \quad
  - \frac{1}{2} 
\frac{\partial^3 \ln \rho}{\partial
  x_i \partial x_j \partial x_k} \, \v_j\,  \v_k  + \frac{1}{\rho} \frac{\partial }{\partial x_j} \left( \rho \frac{\partial
    \v_j}{\partial x_k} \frac{\partial
    \v_k}{\partial x_i}  \right).
\label{eq:EL_uappdx}
\end{align}

The BVP~\eqref{e:bvp_divergence_multi} is obtained by integrating~\eqref{eq:EL_Kappdx} once:
\[
  \frac{\partial}{\partial x_j} (\rho \v_j ) = - (h - \hat{h}) \rho,  
\]  
where $\hat{h} \doteq \int h(x) \rho(x) \ud x$. 
Using this the righthand-side of~\eqref{eq:EL_uappdx} is simplified, and the resulting equation is given by
\notes{reviewer does not understand this conclusion --spm}
\begin{equation}
\frac{\partial }{\partial x_i} \left( \frac{1}{\rho}
  \frac{\partial}{\partial x_j} (\rho u_j ) \right)  = \frac{\partial
  h}{\partial x_i} \hat{h} + \frac{1}{2} 
\frac{\partial }{\partial x_i} \left(  \frac{1}{\rho}
  \frac{\partial^2}{\partial x_j \partial x_k} (\rho \v_j \v_k ) \right).
\label{eq:EL_u}
\end{equation}
The simplification is obtained by first expressing the two terms
involving $h$ in the righthand-side of \eqref{eq:EL_uappdx} as,
\[
h \frac{\partial
    h}{\partial x_i} - \frac{\partial^2 h}{\partial
  x_i \partial x_j} \v_j = \frac{\partial
  h}{\partial x_i} \hat{h} + \frac{1}{\rho} \frac{\partial}{\partial
  x_j} (\rho \v_j ) \frac{\partial }{\partial x_i} \left( \frac{1}{\rho}
  \frac{\partial}{\partial x_k} (\rho \v_k ) \right) + \frac{\partial^2 }{\partial
  x_i \partial x_j} \left( \frac{1}{\rho}
  \frac{\partial}{\partial x_k} (\rho \v_k ) \right) \v_j.
\]
Substituting this in the righthand-side of~\eqref{eq:EL_uappdx} gives the first term $\frac{\partial
  h}{\partial x_i} \hat{h}$ in the righthand-side of~\eqref{eq:EL_u}, 
and four terms involving only $\rho$ and $\v$.  It is a straightforward but
tedious calculation to simplify these four terms into the form expressed as
the second term in the  righthand-side of~\eqref{eq:EL_u}.

It is readily verified, by direct substitution, that~\eqref{eq:EL_u} admits a closed-form solution:
\begin{equation*}
u_j = - \v_j \frac{(h+\hat{h})}{2} + \frac{1}{2} \frac{\partial
  \v_j}{\partial x_k} \v_k.
\label{eq:u_closed_form}
\end{equation*}
This gives~\eqref{eqn:u_intermsof_v*}.
\notes{this is why I don't like the $\tilde u$ notation}

\subsection{Derivation of Equation~\eqref{e:Jpertn}}
\label{appdx:EL_K_u_derivation}

We substitute $\chi_n(x) = x + \v(x,n) \Delta Z_n + u(x,n) \Delta t_n$
in~\eqref{eqn:EL-transport} and obtain explicit expressions for terms
up to order $O(\Delta Z_n),\,O(\Delta t_n)$. Since we are eventually
interested in the limit as $\Delta t_n \rightarrow 0$, we use the It\^{o}'s rule
$(\Delta Z_n)^2 = \Delta t_n$ to simplify the calculations.  
The calculations for the three terms appearing in~\eqref{eqn:EL-transport} are as follows:

\medskip

\noindent (i), The third term is expressed as
\begin{align*}
- \int \rho_{n-1}(x) \left( \Delta Z_{n} -
h(x + \v \Delta Z_n + u \Delta t_n) \Delta t_n \right) \nabla h(x
+ \v \Delta Z_n + u \Delta t_n) \cdot \pertu  (x) \ud x.
\end{align*}
Using Taylor series,
\begin{align*}
h(x + \v \Delta Z_n + u \Delta t_n) & = h(x) + O(\Delta
Z_n,\Delta t_n), \\
 \frac{\partial  h}{\partial x_i} (x + \v \Delta Z_n + u \Delta t_n) & = \frac{\partial h}{\partial x_i}(x) + \frac{\partial^2 h}{\partial
  x_i \partial x_j}(x) \v_j(x) \, \Delta Z_n + O(\Delta t_n,\Delta Z_n \Delta
t_n,\Delta Z_n^2),
\end{align*}
the third term is simplified as
\begin{align*}
& = \left( - \int \rho_{n-1}(x) \frac{\partial h}{\partial x_i} \pertu_i (x)
  \ud x \right)\,\Delta Z_n + \left( \int \rho_{n-1}(x) (h \frac{\partial
    h}{\partial x_i} - \frac{\partial^2 h}{\partial
  x_i \partial x_j} \v_j ) \pertu_i (x)
  \ud x \right)\,\Delta t_n \\
& \quad \quad \quad \quad+ O(\Delta t_n^2, \Delta Z_n \Delta t_n, \Delta Z_n^3).
\end{align*}

\medskip

\noindent (ii) The second term in~\eqref{eqn:EL-transport} is
similarly simplified as
\begin{align*}
& - \int \rho_{n-1}(x) \nabla \ln \Bigl( \rho_{n-1}\bigl(x + \v(x) \Delta Z_n + u(x) \Delta
  t_n \bigr) \Bigr) \cdot \pertu  (x) \ud x\\
& = - \int \rho_{n-1}(x)  \frac{\partial }{\partial x_i} \ln (\rho_{n-1}) \, \pertu_i
(x) \ud x  + \left( - \int \rho_{n-1}(x) \frac{\partial^2 \ln \rho_{n-1}}{\partial
  x_i \partial x_j} \, \v_j \, \pertu_i
(x) \ud x \right) \, \Delta Z_n \\
& \quad \quad\quad\quad  + \left( - \int \rho_{n-1}(x) ( \frac{\partial^2 \ln \rho_{n-1}}{\partial
  x_i \partial x_j} \, u_j + \frac{1}{2}
\frac{\partial^3 \ln \rho_{n-1}}{\partial
  x_i \partial x_j \partial x_k} \, \v_j\,  \v_k) \, \pertu_i
(x) \ud x \right) \, \Delta t_n \\ 
& \quad \quad\quad\quad +  O(\Delta t_n^2,\Delta Z_n \Delta
  t_n,\Delta Z_n^3).
\end{align*}

\medskip

\noindent (iii) Finally, for the remaining term in~\eqref{eqn:EL-transport},
\begin{align*}
& - \int \rho_{n-1}(x)
\text{tr} \left( D \chi_n^{-1}(x) D \pertu (x) \right) \ud x \\
& = - \int \rho_{n-1}(x)
\text{tr} \left( (I + D \v(x) \Delta Z_n + D u(x) \Delta
  t_n)^{-1} D \pertu (x) \right) \ud x \\
& = \int \nabla \rho_{n-1}(x) \cdot \pertu (x) \ud x + \left( \int \rho_{n-1}(x)
  \tr(D\v(x)  D \pertu (x)) \ud x \right)\, \Delta Z_n \\
& \quad \quad\quad\quad + \left( \int \rho_{n-1}(x)
  (\tr(Du(x)  D \pertu (x)) - \tr((D\v)^2(x) D \pertu (x))) \ud x
\right)\, \Delta t_n \\
& \quad \quad\quad\quad +   O(\Delta t_n^2,\Delta Z_n \Delta
  t_n,\Delta Z_n^3).
\end{align*}
Now, the terms with trace are simplified by using integration by
parts, e.g., 
\begin{align*}
\int \rho_{n-1}(x) \tr(D\v(x)  D \pertu (x)) \ud x & = \int \rho_{n-1}(x)
\frac{\partial \v_j}{\partial x_i} \frac{\partial \pertu_i}{\partial x_j} \ud
x\\
& = - \int \frac{\partial}{\partial x_j} \left( \rho_{n-1}(x) \frac{\partial
    \v_j}{\partial x_i} \right)\, \pertu_i(x) \ud x.
\end{align*}
As a result, the final term is given by 
\begin{align*}
& = \int \frac{\partial \rho_{n-1}}{\partial x_i}(x) \pertu_i(x) \ud x + \left(- \int \frac{\partial}{\partial x_j} \left( \rho_{n-1}(x) \frac{\partial
    \v_j}{\partial x_i} \right)\, \pertu_i(x) \ud x \right)\, \Delta Z_n \\
& \quad \quad + \left( - \int \frac{\partial }{\partial x_j}
 \left( \rho_{n-1}(x) \frac{\partial
    u_j}{\partial x_i} \right)\, \pertu_i(x) \ud x + \int
\frac{\partial }{\partial x_j} \left( \rho_{n-1}(x) \frac{\partial
    \v_j}{\partial x_k} \frac{\partial
    \v_k}{\partial x_i}  \right)\, \pertu_i(x) \ud x
\right)\, \Delta t_n \\
& \quad \quad\quad\quad+  O(\Delta t_n^2,\Delta Z_n \Delta
  t_n,\Delta Z_n^3)
\end{align*}

\medskip

Collecting the three terms, the E-L equation~\eqref{eqn:EL-transport} is given by 
\[
\delta I_n(\chi_n) \cdot \pertu = E_1 + E_z \Delta Z_n + E_\Delta \Delta t_n + O(\Delta t_n^2,\Delta Z_n \Delta
  t_n,\Delta Z_n^3),
\]
where $E_1$ is the O(1) term given by
\begin{align*}
E_1 & = - \int \rho_{n-1}(x)  \frac{\partial }{\partial x_i} \ln (\rho_{n-1}) \, \pertu_i
(x) \ud x + \int \frac{\partial \rho_{n-1}}{\partial x_i}(x) \pertu_i(x) \ud
x =0,
\end{align*}
$E_z$ is the $O(\Delta Z_n)$ term given in \eqref{e:E2}, and $E_\Delta$ is the $O(\Delta t_n)$ term given in \eqref{e:E3}.

\subsection{Proof of Theorem~\ref{thm:exact}}
\label{apdx:pf_Kushner}

We first assume that $U_t^i$ is admissible.  In this case, the
evolution of $p(x,t)$ is according to the forward equation:
\begin{equation}
\begin{aligned}
\ud p = 
- \nabla \cdot (p\v) \ud Z_t
- \nabla \cdot (pu) \ud t +
\frac{1}{2}\sum_{l,k=1}^d \frac{\partial^2}{\partial x_l \partial
  x_{k}} \left( p \v_l \v_{k} \right) \ud t.
  \end{aligned}
\label{eqn:mod_FPK}
\end{equation}

To prove that the filter is exact, one needs to show that with the choice of $\{u,\v\}$ given
by~\eqref{e:bvp_divergence_multi}-\eqref{eqn:u_intermsof_v*},  we have
$\ud p(x,t) = \ud p^*(x,t)$, for all $x$ and $t$,  in the sense that
they are defined by identical stochastic differential equations.   Recall $\ud p^*$ is
defined according to the K-S equation~\eqref{eqn:Kushner_eqn}.  The
strong form of evolution equations is used for notational convenience.
The proof with the weak form is entirely analogous, by using
integration by parts.

Recall that the gain function $\v$ is a solution of Poisson's equation,
\begin{equation}
\nabla \cdot (p\v) = -p(h-\hat{h}) \, .
\label{eqn:bvp_multi_Z_matrix}
\end{equation}
On multiplying both sides of~\eqref{eqn:u_intermsof_v*} by $-p$, we obtain
\begin{equation}
\begin{aligned}
-up & = \frac{1}{2} \v (h-\hat{h})p - \w p + p \v \hat{h} \\
      & = - \frac{1}{2} \v
      \nabla \cdot (p\v) - \w p + p \hat{h}\v 
\end{aligned}
 \label{eqn:interm_up}
\end{equation}
where~\eqref{eqn:bvp_multi_Z_matrix} is used to obtain the second
equality.  Denoting $E:= \frac{1}{2} \v \,\nabla \cdot (p\v)$, a direct calculation shows that
\begin{equation*}
E_{l} +  \w_{l}p = \frac{1}{2}\sum_{k=1}^d \frac{\partial }{\partial x_k}
\left( p[\v \v^T]_{lk} \right).
\end{equation*}
\notes{$[\v \v^T]_{lk} = \v_l \v_k$, so I sometimes use the latter -- spm}

Substituting this in~\eqref{eqn:interm_up}, on taking the divergence
of both sides, we obtain
\begin{align}
-\nabla \cdot (pu) +\frac{1}{2}\sum_{l,k=1}^d \frac{\partial^2}{\partial x_l \partial
  x_{k}} \left( p \v_l \v_k \right) &= \nabla \cdot (p\v) \hat{h}.
\label{eqn:FPK_23}
\end{align}
Using~\eqref{eqn:bvp_multi_Z_matrix} and~\eqref{eqn:FPK_23} in the forward equation~\eqref{eqn:mod_FPK},
\begin{align*}
\ud p & =  
( h-\hat{h} )(\ud Z_t - \hat{h} \ud t)p.
\end{align*}
This is precisely the K-S equation~\eqref{eqn:Kushner_eqn}, as desired.

\medskip

Finally, we show that $U_t^i$ is admissible.  This follows
from~\Prop{prop:exist_uniq} and Theorem~\ref{thm:thm1}.  The posterior
distribution $p^*$ is the limit of the minimizer sequence
$\{\rho_n\}$, where $\rho_n$ satisfies {\bf PI}($\bar{\lambda}$) and
$\bar{\lambda}>0$ for all $n$.  By Theorem~\ref{thm:thm1}, a
unique solution $\v(x,t)=\nabla\phi(x,t)$ exists for each
$p(x,t)=p^*(x,t)$.   The a priori
bounds~\eqref{eqn:bound1}-\eqref{eqn:bound2} are used to show that
\begin{align*}
\Expect[|\v|^2] &
	 \le \Expect\left[  \frac{1}{\bar\lambda} \int |h(x)|^2
p(x,t) \ud x \right] <\infty,
\\[.3cm]
\Expect[ |u| ] & \le  \Expect\left[  \left( \frac{1}{\bar\lambda} +
  C(\bar\lambda;p)^{1/2} \right)  \int \left( |h(x)|^2 + |\nabla h|^2 \right)
p(x,t) \ud x \right] <\infty,
\end{align*}
where the expression for $C(\bar\lambda;p)$ appears in
Theorem~\ref{thm:thm1}, and we have used the fact that $h,\nabla h\in L^{\infty}$.  That is, the
resulting control input in the feedback particle filter is admissible.
\qed

\bibliography{gradient_flow_FINAL.bbl}

\end{document}